\documentclass[11pt]{amsart}

\usepackage{amsmath}
\usepackage[english]{babel}
\usepackage{amssymb}
\usepackage{enumitem}
\usepackage[initials]{amsrefs}
\usepackage[all]{xy}
\usepackage{bbm}
\usepackage{ulem}
\usepackage{float}

\usepackage[pdftex,colorlinks,citecolor=blue]{hyperref}

\usepackage{fancyhdr}

\usepackage{amscd}
\usepackage{amsfonts}
\usepackage{mathrsfs}
\usepackage{setspace}
\usepackage{version}
\usepackage{mathtools}

\SelectTips{cm}{}

\usepackage{graphicx}
\usepackage{tikz-cd}

\newcommand{\bbN}{{\mathbb N}}
\newcommand{\bbQ}{{\mathbb Q}}

\newcommand{\bbR}{{\mathbb R}}
\newcommand{\bbZ}{{\mathbb Z}}

\newcommand{\id}{\operatorname{id}}

\newcommand{\diag}{\operatorname{diag}}

\newcommand{\diam}{\operatorname{diam}}

\newcommand{\SL}{\operatorname{SL}}

\newcommand{\Aut}{\operatorname{Aut}}

\newcommand{\Sym}{\operatorname{Sym}}

\newcommand{\PSL}{\operatorname{PSL}}
\newcommand{\PGL}{\operatorname{PGL}}
\newcommand{\GL}{\operatorname{GL}}

\newcommand{\efface}[1]{}

\newtheorem{mthm}{Theorem}

\newtheorem{theorem}{Theorem}[section]
\newtheorem{lemma}[theorem]{Lemma}

\newtheorem{corollary}[theorem]{Corollary}
\newtheorem{cor}[theorem]{Corollary}
\newtheorem{proposition}[theorem]{Proposition}
\newtheorem{prop}[theorem]{Proposition}

\theoremstyle{definition}
\newtheorem{definition}[theorem]{Definition}

\newtheorem{example}[theorem]{Example}

\newtheorem{remark}[theorem]{Remark}

\numberwithin{equation}{section}

\begin{document}
\title[Measure rigidity for horospherical subgroups of  $\Aut(T)$]{Measure rigidity for horospherical subgroups of groups acting on trees}

\author{Corina Ciobotaru}
\author{Vladimir Finkelshtein}
\author{Cagri Sert}
\address{Department of Mathematics, University of Fribourg, Chemin du Mus\'{e}e 23, 1700 Fribourg, Switzerland}
\email{corina.ciobotaru@gmail.com}
\address{Mathematisches Institut, Georg-August-Universit\"{a}t G\"{o}ttingen, Bunsenstra\ss e 3-5, 37073 G\"{o}ttingen, Germany}
\email{filyok@gmail.com}
\address{Departement Mathematik, ETH Z\"{u}rich, R\"{a}mistrasse 101, Z\"{u}rich, Switzerland}
\email{cagri.sert@math.ethz.ch}

\begin{abstract}
We prove analogues of some of the classical results in homogeneous dynamics in nonlinear setting. Let $G$ be a closed subgroup of the group of automorphisms of a biregular tree and $\Gamma\leq G$ a discrete subgroup. 
For a large class of groups $G$, we give a classification of the probability measures on $G/\Gamma$ invariant under horospherical subgroups.  When $\Gamma$ is a cocompact lattice, we show the unique ergodicity of the horospherical action. Moreover, we prove Hedlund's theorem for geometrically finite quotients. Finally, we show equidistribution of large compact orbits.
\end{abstract}

\subjclass[2010]{22D40, 20E08}

\maketitle

\section{Introduction}

The study of unipotent dynamics on quotients of linear algebraic groups by discrete subgroups is closely related to numerous problems in number theory and geometry.
In the 70's, inspired by this connection, Raghunathan formulated a conjecture concerning the closures of orbits of unipotent subgroups in the context of Lie groups. Today, unipotent dynamics are well-understood for linear algebraic groups over local fields due to seminal works of Ratner, Dani, Margulis and many others. In this paper, following a geometric analogy, we study a nonlinear counterpart of this setting. 

To explain the geometric analogy, recall that a reductive algebraic group over a non-archimedean local field acts on a natural polysimplicial complex called the Bruhat--Tits building. This is an analogue of the symmetric space of a real semisimple Lie group. For a rank-one simple algebraic group (e.g. $\SL_2$), the associated Bruhat--Tits building is a biregular tree. The horospherical stabilizer of a point in the boundary of the tree is a compact extension of the corresponding unipotent subgroup.

Once this point of view is adopted, many natural questions arise. More precisely, let $T$ be a biregular tree and $\xi \in \partial T$. Let $G \leq \Aut(T)$ be a closed subgroup, $\Gamma\leq G$ a discrete subgroup, and $G^0_\xi:=\{g\in G \, |\, g \xi=\xi \; \text{and} \; g \; \text{is elliptic}\}$ the horospherical stabilizer of $\xi$ in $G$. What can be said about the dynamics of the $G^0_\xi$-action on $G/\Gamma$? What are the analogues of number theoretical applications of linear homogeneous dynamics? 

This paper addresses the former question. In the linear setting, the study of unipotent dynamics mostly goes as follows: first, one classifies invariant measures, which leads to equidistribution results, which in turn allows understanding of the orbit closures. 
In this work we answer these questions in different degrees of generality.

\subsection{Main results}

\subsubsection{Measure classification}
\label{sec::111}
We are able to obtain a measure classification result for large subgroups $G$ of $\Aut(T)$. Largeness of $G$ is reflected in two hypotheses below: we require that $G$ satisfies Tits' independence property (\S \ref{subsub.Tits.indep}) and  condition (flip). Tits' independence insures that there are sufficiently many rotations in $G$; it was introduced by Tits  to prove simplicity of large subgroups of $\Aut(T)$. Condition (flip) is a transitivity assumption on the action of $G$ on the boundary of the tree $\partial T$; it means that for every triple of distinct ends $\xi_1,\xi_2,\xi_3 \in \partial T$, there is an element $g \in G$ fixing $\xi_1$ and satisfying $g \xi_i= \xi_j$ for $i \neq j \in \{2,3\}$. It is clearly implied by $3$-transitivity of $G$-action on $\partial T$ but not vice versa (see \S \ref{subsub.examples}).

We are now ready to state our main result. 

\begin{mthm}[Measure classification]\label{thm.measure.classification}
Let $T$ be a $(d_0,d_1)$-biregular tree with $d_0,d_1 \geq 3$ and $G$ be a closed, topologically simple subgroup of $\Aut(T)$, satisfying Tits' independence property and condition (flip). Let $\xi \in \partial T$ and $G_\xi^0 \leq G$ be the horospherical stabilizer of $\xi$. Let $\Gamma$ be a discrete subgroup of $G$, denote $X=G/\Gamma$ and let $\mu$ be a $G_\xi^0$-invariant and ergodic Borel probability measure on $X$. Then, either $\mu$ is a $G^0_\xi$-homogeneous measure, or $\Gamma$ is a lattice in $G$ and $\mu=m_X$ is the unique $G$-invariant probability measure on $X$.
\end{mthm}
Recall that a probability measure $\mu$ on $X$ is said to be $G^0_\xi$-homogeneous if it is the unique $G^0_\xi$-invariant Borel probability measure on a closed $G^0_\xi$-orbit.

Groups satisfying the hypotheses of previous theorem include $\Aut(T)$, with $d_0 \neq d_1$,  $\Aut(T_d)^+$ and universal groups $U(F)^+$ in the sense of Burger--Mozes (see \S \ref{subsub.examples}), where the permutation group $F\leq \Sym(d)$ satisfies the corresponding condition (flip). Further examples are discussed in \S \ref{subsub.examples}. 

The proof of this theorem relies on an analogue of Ratner's drift/transverse divergence argument which produces additional invariance of the probability measure $\mu$: to set this argument in our case, we make use of geometric versions of some classical decompositions of linear groups. Controlling the analogue of the time change in the drift argument is the subtle part; to do this we need to show existence and uniqueness of solutions of certain equations coming from Bruhat decomposition of $G$ - this is where Tits' independence and (flip) are used. Once the drift argument is executed, the rest of the proof follows a strategy due to Ghys \cite{ghys} in his alternative proof of the measure classification for $\SL(2,\mathbb{R})$-quotients.

Let us list a few remarks on the hypotheses of Theorem \ref{thm.measure.classification}.

\begin{remark}\label{remark.intro1}
1. We note that if $G$ is topologically simple and satisfies Tits' independence property, then $G$ is necessarily nonlinear (\cite[Corollary R]{caprace-reid-willis}).\\[-9pt]

2. Geometric correspondents of unipotent subgroups of linear groups are contraction groups. 
In simple algebraic groups, all contraction groups are known to be closed, but many nonlinear examples have no closed contraction groups. The group $G_\xi^0$ is the closure of the corresponding contraction group in $G$, whenever $G$ satisfies Tits' independence property (see Lemma \ref{lem::closure_contraction_gr}) and, hence, is the natural analogue of the unipotent flow in our setting. We remark that, unlike in the real case, any finite collection of elements in  $G_\xi^0$ generates a precompact group, which does not lead to interesting dynamics from the perspective of this article. \\[-9pt] 
    
3. The topological simplicity assumption is crucial for us, since it ensures the Howe--Moore property, which is one of the central ingredients of our proof. Moreover, topological simplicity, together with $2$-transitivity of $G$ on $\partial T$, guarantees that there are no proper closed unimodular subgroups strictly containing $G_\xi^0$, allowing us to prove that the only possible $G_\xi^0$-invariant probability measures other than $m_X$ are $G_\xi^0$-homogeneous. 

\end{remark}

\subsubsection{Unique ergodicity}
${}$

When the lattice is cocompact, the measure classification can often be proved using different techniques. Indeed, our result for compact quotient covers a larger family of groups $G \leq \Aut(T)$ and parallels Furstenberg's theorem \cite{furstenberg} on unique ergodicity of horocycle flows.

\begin{mthm}[Unique ergodicity of compact quotients]\label{thm.unique.ergo}
Let $T$ be a $(d_0,d_1)$-biregular tree with $d_0,d_1 \geq 3$. Let $G$ be a non-compact, closed, topologically simple subgroup of $\Aut(T)$ acting transitively on $\partial T$ and satisfying Tits' independence property. Let $\xi \in \partial T$ and $G^0_\xi \leq G$ be the horospherical stabilizer of $\xi$. Let $\Gamma$ be a uniform lattice in $G$ and  $X=G/\Gamma$. Then, the $G^0_\xi$-action on $X$ is uniquely ergodic.
\end{mthm}
An immediate topological consequence is the following:

\begin{corollary}\label{corol.intro.minimal}
The action of $G^0_\xi$ on the compact set $X$ is minimal, i.e. every $G^0_\xi$-orbit is dense.
\end{corollary}

The main ingredient in the proof of Theorem \ref{thm.unique.ergo} is the Howe--Moore property: it allows us to employ an orbit thickening argument, due to Margulis \cite{margulis.thesis}. Similar ideas are used in the proof of unique ergodicity in a general setting by Bowen--Marcus \cite{bowen} and Ellis--Perrizo \cite{ellis-perrizo} (see also \cite{mohammadi}). A common assumption in these results is that conjugation by a certain hyperbolic element contracts compact sets in unipotent groups to identity. However, this fails for the group $G_\xi^0$ in our setting. This issue is dealt with by using Tits' independence property and the structure of uniform lattices.

\begin{remark}
A closed group $G \leq \Aut(T)$ is non-compact and transitive on $\partial T$ if and only if it is $2$-transitive on $\partial T$  (\cite[Lemma 3.1.1]{BM00b}). Hence, Theorem \ref{thm.unique.ergo} assumes $2$-transitivity of $G$-action on the boundary, a weaker hypothesis than condition (flip) used in Theorem \ref{thm.measure.classification}. It is easy to see that (flip) implies $2$-transitivity, but not vice versa.\\[-9pt]
\end{remark}

\subsubsection{Geometrically finite quotients}\label{subsub.intro.geo.fin}
The family of lattices in $\Aut(T)$ is much less tractable than their linear counterparts. 
For tree lattices, the ``All Quotients Theorem" (\cite[Theorem 4.17]{BL}) holds: a consequence in a regular tree $T_d$ is that ``every combinatorially allowable" geometric cusp structure occurs for some $\Gamma \setminus T_d$ and the fundamental group of the graph $\Gamma \setminus T_d$ can be infinite-countably generated, contrary to the lattices in simple linear algebraic groups, for which the corresponding fundamental group is always finitely generated.
 Moreover, by work of Bekka--Lubotzky \cite{bekka-lubotzky} answering by the negative a question of Margulis, there is a lattice $\Gamma$ in $G=\Aut(T_d)$ such that the regular representation of $G$ on $L^2(G/\Gamma)$ does not have a spectral gap, in contrast to the linear case.

In the next theorem, we restrict our attention to the family of geometrically finite lattices. These are lattices $\Gamma\leq G$ such that the quotient graph $\Gamma \backslash T$ is, in particular, a union of a finite graph and finitely many rays (for a precise definition, see \S \ref{subsec.geo.fin}). By work of Lubotzky \cite{lubotzky.gafa}, this class of groups contains all the algebraic examples: for any non-archimedean local field $\mathrm{k}$ and semisimple linear algebraic $\mathrm{k}$-group $\mathbb{G}$ of $\mathrm{k}$-rank one, any lattice $\Gamma$ in $\mathbb{G}(\mathrm{k})$ is a geometrically finite lattice of $\Aut(T)$ where $T$ is the Bruhat--Tits tree associated to $\mathbb{G}(\mathrm{k})$.

For this family, we give a complete description of the $G^0_\xi$-orbit closures as in Hedlund's theorem \cite{hedlund} for $\SL(2,\mathbb{R})$-quotients.

\begin{mthm}[Geometrically finite quotients]\label{thm.geo.fin}
Let $T$ be a $(d_0,d_1)$-biregular tree with $d_0,d_1 \geq 3$, $G$ be a non-compact, closed and topologically simple subgroup of $\Aut(T)$ acting transitively on $\partial T$, $\Gamma$ a geometrically finite lattice in $G$ and $\xi \in \partial T$.  \\[3pt]
1. A $G^0_\xi$-orbit in $X$ is either compact or dense.\\[1pt]
2. Let $a \in G$ be a hyperbolic element of  translation length $2$ with fixed point $\xi \in \partial T$ and suppose that $\Gamma$ has $k$ cusps. Then, there exist $x_1,\ldots,x_k \in X$ such that the family of compact $G^0_\xi$-orbits consists precisely of the orbits of $a^i x_j$ for $i \in \mathbb{Z}$ and $j=1,\ldots,k$.
\end{mthm}


In the proof of this result, the characterization of compact orbits requires understanding of the geometry of $\Gamma$-action on $T$. This was studied in detail by Paulin \cite{paulin.geo.fin} and we shall make extensive use of Paulin's result for this part of the proof. To prove the density of non-compact orbits, we employ a variant of Margulis' orbit-thickening argument using, in particular, the Howe--Moore property of $G$ proved by Burger--Mozes \cite{BM00a} (see also \cite{lubotzky-mozes}). To apply this argument, we make further use of Paulin's work \cite{paulin.geo.fin} to understand the geometry of the action of $a$, which is reminiscent of the geodesic flow.


\begin{remark}
In Theorem \ref{thm.geo.fin}, we do not assume that $G$ has Tits' independence property; so the theorem applies to any group $G=\mathbb{G}(\mathrm{k})$ of $\mathrm{k}$-points of a semisimple linear algebraic $\mathrm{k}$-group $\mathbb{G}$ of $\mathrm{k}$-rank one. A slight difference of our setup from the classical setting of unipotent dynamics over local fields is that the group $G^0_\xi$ that we consider, corresponds to a compact extension of the horospherical subgroups of such linear groups. This can be seen, for example, by the Levi decomposition (\S \ref{subsub.decompositions}).
\end{remark}

\begin{remark}The study of the discrete geodesic flow on geometrically finite quotients has proven to be very fruitful. For example, for $\Gamma=\PGL_2(F_q[t])$ the fundamental domain $\Gamma \backslash T$ is a ray (see \cites{bass-serre, BL}). This 
allows one to translate problems on continued fraction expansion and Diophantine approximation in  $F_q((t^{-1}))$ to questions about dynamics in this quotient, sometimes called the modular ray due to its similarities with the modular surface (see \cites{paulin, broise-paulin.1, hersonsky-paulin,BPP.cras,BPP.book,KLNN,paulin-shapira,paulin.geo.fin}).
\end{remark}

\subsubsection{Equidistribution of compact orbits}
It follows from an observation of Mostow \cite{mostow} (see Proposition \ref{prop.mostow}) that given a compact orbit $G^0_\xi x$ in $X$, the union of its translates $a^i G^0_\xi x$ for $i \in \mathbb{Z}$, by an appropriate hyperbolic element $a$, is dense in $X$. In our final result, we note a finer property of these translates, namely we show they equidistribute with respect to $m_X$ as their volume tends to infinity. 

Fixing a Haar measure on $G^0_\xi$, we endow its orbits with the corresponding orbital measures. Let $a$ be a hyperbolic element with attracting fixed point $\xi \in \partial T$. The element $a$ normalizes the group $G^0_\xi$ and its conjugation action expands the Haar measure of $G^0_\xi$. In particular, the volume of translates by $a$ of compact $G^0_\xi$-orbits grows to infinity.

\begin{proposition}[Equidistribution of translates of compact orbits]\label{prop.equidist.compact.orbits}
Let $G$ be as in Theorem \ref{thm.unique.ergo} and $\Gamma \leq G$ a lattice. Let $a \in G$ be as above and $x \in X$ be such that its $G^0_\xi$-orbit is compact. Then, the sequence of compact orbits $G^0_\xi  a^i x$ equidistributes to $m_X$ as $i \to +\infty$.
\end{proposition}

The analogue of this result for finite volume hyperbolic surfaces was originally proven by Sarnak \cite{sarnak} with a convergence rate using number theoretic methods. Later, Eskin--McMullen \cite{eskin-mcmullen} gave a short non-effective proof and we follow closely their approach. We mention that after the appearance of an earlier version of this article, we were informed that this result is also a consequence of Theorem 10.2 in the recent book of Broise--Alamichel--Parkkonen--Paulin \cite{BPP.book}.

In the setting of geometrically finite quotients, Proposition \ref{prop.equidist.compact.orbits} takes a particularly nice form and gives a complete description of topology of $G^0_\xi$-invariant and ergodic probability measures (see also \S \ref{subsec.geo.fin}).

\begin{corollary}\label{corol.equidist.geo.fin}
Let $G$ be as in Theorem \ref{thm.unique.ergo}, $\Gamma \leq G$ a geometrically finite lattice, and $\xi \in \partial T$. Then, any sequence of compact $G_\xi^0$-orbits with increasing volumes equidistributes to the Haar measure on $G/\Gamma$.
\end{corollary}

Finally, we mention that in a forthcoming work \cite{CFS.rec}, we prove Dani--Margulis type quantitative recurrence results for the action of the horospherical groups on $X$ and deduce equidistribution of dense orbits in Theorem \ref{thm.geo.fin}.

\subsubsection{Organization of the article} We set our notations in \S \ref{subsection.notation}. In the rest of Section \ref{section.generalities}, we collect some preliminary material regarding groups acting on trees, ergodic theory for amenable groups and differentiation of measures. In Section \ref{sec.unique.ergo}, we prove Theorem \ref{thm.unique.ergo}. Section \ref{sec.main.proof} is devoted to the proof of Theorem \ref{thm.measure.classification} assuming  key Proposition \ref{prop.variable.change.RN} which itself is proven in Section \ref{sec.proof.of.prop}. In Section \ref{sec.geo.fin}, we discuss the geometrically finite lattices and prove Theorem \ref{thm.geo.fin}. Finally, in Section \ref{sec.equidist} we prove Proposition \ref{prop.equidist.compact.orbits}.

\subsection*{Acknowledgements}
The authors are thankful to Marc Burger for asking the original question that led to the results of this paper, and are grateful to Pierre-Emmanuel Caprace, Manfred Einsiedler and Jean-Fran\c cois Quint for useful discussions and to Shahar Mozes for pointing out the reference \cite{ellis-perrizo}. V.F is supported by ERC Consolidator grant 648329 (GRANT). C.S. is supported by SNF grant 178958.

\setcounter{tocdepth}{1}

\section{General facts}\label{section.generalities}

\subsection{Basic Notation}\label{subsection.notation}
We gather here the notation that is fixed throughout the article. 
Let $T$ be a $(d_0,d_1)$-biregular tree with $d_0,d_1 \geq 3$. Denote by $\Aut(T)$ the group of automorphisms of $T$ acting without edge inversion. 
When $d_0 =d_1=d \geq 3$,  $T=T_d$ is a $d$-regular tree in which case we denote $\Aut(T)=:\Aut(T_d)^+$. 



Unless explicitly mentioned, $G$ will denote a closed, topologically simple subgroup of $\Aut(T)$ that acts $2$-transitively on the boundary $\partial T$. The subgroup $\Gamma\leq G$ is always discrete, and $X:=G/\Gamma$. 

Consider the graph metric on $T$. For $v,w$ vertices of $T$, $[v,w]$ denotes the geodesic path in the tree $T$ between the vertices $v,w$. For $\xi \in \partial T$, $[v,\xi)$ is the geodesic ray  from $v$ to $\xi$. For a vertex $v\in T$ and $n\in \bbN$, $B(v,n)$ is a metric ball of radius $n$ around $v$. 
For a directed edge $[v,w]\subset T$ we denote by $T_{[v,w]}$ the unique maximal (for the inclusion) subtree of $T$ that contains the edge $[v,w]$, but does not contain any neighbors of $v$ other than $w$. 

Let $\xi \in \partial T$. By Tits \cite{Ti70}*{Prop. 3.4} $G$ admits hyperbolic elements. By transitivity of $G$ on the boundary $\partial T$, we can fix a hyperbolic element $a \in G$ with translation length $2$ with $\xi$ as its attracting fixed point. We remark that $G$ acting on $T$ without edge inversion and $2$-transitively on $\partial T$ does not contain hyperbolic elements of translation length $1$. Denote by $\xi_- \in \partial T$ the repelling fixed point of $a$ and enumerate by $\{x_i\}_{i \in \bbZ} \subset T$ the vertices of the geodesic line $(\xi_-, \xi)$. As a convention, we choose the vertex $x_0$ so that it has valency $d_0$. The meaning of symbols $a, \xi, \xi_-, x_i$'s is fixed for the rest of the article, unless otherwise specified.

For a subgraph $D$ of $T$ denote by 
\[G_D:=\{ g\in G \; | \; g(v)=v, \text{ for all } v\in D \}\]
the pointwise stabilizer of $D$. For $\xi \in \partial T$, let $$G_{\xi}:=\{ g \in G  \;\vert \; g(\xi)= \xi \}$$  be the stabiliser of $\xi$ and
\[G_{\xi}^{0}:=\{ g \in G  \;\vert \; g(\xi)= \xi, \; g \text{ elliptic} \}\]
be the \textbf{horospherical stabiliser of $\xi$}. 
Denote  $$M:=G_{\xi_-}^0 \cap G_{\xi}^0.$$

Having fixed $a$, we define the \textbf{positive contraction group} 
\[ 
U^{+}:= \{ g \in G \; | \; \lim\limits_{n \to \infty} a^{-n}ga^{n}=\id \},
\]
and, similarly, the \textbf{negative contraction group}
\[U^{-}:=\{ g \in G \; | \; \lim\limits_{n \to \infty} a^{n}ga^{-n}=\id \}.\]

\begin{example}
\label{ex::decomp_on_SL}
Let $G=\PSL(2, \bbQ_p)$. It is well-known that $G$ acts on a $(p+1)$-regular tree (see e.g. \cite{bass-serre}) and the diagonal matrix $a=\diag(p, p^{-1})$ in $G$ is a hyperbolic element of translation length $2$. In this case 
\[  G^{0}_{\xi}=\begin{pmatrix}
\bbZ_p^* & \mathbb{Q}_p \\
0 & \bbZ_p^*
\end{pmatrix},\,  \; \; \;   
 U^+ = \begin{pmatrix}
1 &  \mathbb{Q}_p \\
0 & 1
\end{pmatrix}.
\]

We remark that in this example $G$ does not satisfy Tits' independence property (see Remark \ref{remark.intro1}). 

\end{example}

All our probability measures are assumed to be Borel. For a closed subgroup $H \leq G$, $m_H$ will always denote a left Haar measure on $H$. We will specify the normalizations when needed.

Finally, we recall the following notion from the introduction
\begin{definition}\label{equ::flip}
Let $G$ be a subgroup of $\Aut(T)$. We say that $G$ satisfies \textbf{condition (flip)} if for every triple of distinct ends $\xi_1,\xi_2,\xi_3 \in \partial T$ there is an element $g \in G^{0}_{\xi_1}$ with $g \xi_2= \xi_3$  and $g \xi_3= \xi_2$.
\end{definition}

\subsection{Generalities on groups acting on trees}
\subsubsection{Classical decompositions}\label{subsub.decompositions}
 
The following two decompositions are well-known (see e.g. \cite[Theorem 1.5.2)]{Chou}) and the third one is due to \cite{BW04}.

\begin{enumerate}

\item
\label{Bruhat_decomposition}
\textbf{Bruhat decomposition:}
$$G= G_{\xi}^{0} w G_{\xi}  \sqcup G_{\xi}= G_{\xi_-}^{0}G_{\xi}\;  \sqcup \;  w^{-1}G_{\xi},$$
where  $w \in G$ is an elliptic element with $w(\xi)=\xi_-$ and $w(\xi_-)=\xi$.

\item
\label{AN_decomposition}
\textbf{AN decomposition:}  $$G_{\xi}^{0} \lhd G_{\xi} \text{ and } G_{\xi}= a^\bbZ G_{\xi}^{0}.$$ 

\item
\label{Levi_decompositions}
\textbf{Levi decomposition:}
$$G_{\xi}^{0}= U^{+} M \; , \; G_{\xi_-}^{0}= U^{-} M .
$$

\end{enumerate}

\subsubsection{Tits' independence property}\label{subsub.Tits.indep}
 
An important hypothesis we impose on the group $G$ is the following property introduced by Tits~\cite{Ti70}. It guarantees the existence of enough rotations in the group $G \leq \Aut(T)$.

\begin{definition}
\label{def::Tits_prop}
Let $T$ be a locally finite tree and $G \leq \Aut(T)$. We say that $G$ has \textbf{Tits' independence property} if for every edge $[x,y] \subset T$, we have the equality $G_{[x,y]}=G_{T_{[x,y]}}G_{T_{[y,x]}}$. 
\end{definition}

By \cite[Lemma 10]{Amann}, when $G$ is closed, $G$ satisfies Tits' independence property if and only if for any subtree $S \subset T$, having at least one edge, we have  
$$G_S = \prod_{x\in S} G_{S_x}, \quad \text{ where for } x\in S, \; S_x : = \bigcup_{y\in S \cap B(x,1)} T_{[x,y]}.$$


\subsubsection{Examples}\label{subsub.examples}

Explicit examples of closed subgroups $G \leq \Aut(T_d)^+$ with Tits' independence property are the universal groups introduced by Burger--Mozes in \cite[Section  3]{BM00a}. We recall the definition.

\begin{definition}
\label{def::legal_color}
Let $\iota: E(T_d) \to \{1,...,d\} $, where   $E(T_d)$ is the set of unoriented edges of $T_d$. The set $E(x)\subset E(T_d)$ consisting of edges containing the vertex $x\in T_d$ is called the star of $x$. We say that $\iota$ is a legal coloring of $T_d$ if for every vertex $x\in T_d$ the map $i|_{E(x)}$ is surjective. \end{definition}

Given a legal coloring  $\iota$,  for each $g\in \Aut(T_d)$ and $x \in T_d$, we denote $g_*(x):= \iota \circ g \circ (\iota \vert_{E(x)})^{-1} \in \Sym(d)$.

Let $F \leq \Sym(d)$ and let $\iota$ be a legal coloring of $T_d$. The \textbf{universal group}, with respect to $F$ and $\iota$, is defined  as
$$ U(F):= \{g \in \Aut(T_d) \; \vert \; g_*(x) \in F \text{ for every } x \in T_d \}. $$
In fact, the group $U(F)$ is independent of the legal coloring $\iota$. We denote by $U(F)^+$ the subgroup generated by the edge-stabilizing elements of $U(F)$.  When $F$ is transitive and generated by its point stabilizers, we have  $U(F)^{+} = U(F) \cap \Aut(T_d)^{+}$. 

The groups $U(F)^+$ are closed, abstractly simple (\cite[Proposition 3.2.1]{BM00a}), and $k$-transitive on $\partial T_d$ whenever  $F$ is a $k$-transitive permutation group. Furthermore, they satisfy Tits' independence property (\cite[Proposition 48]{Amann}).
For $U(F)^{+}$ condition (flip) (see Definition \ref{equ::flip}) is clearly implied by the 3-transitivity of $F$. On the other hand,  (flip) and the 3-transitivity are not equivalent in general. Indeed, consider a finite field $\mathrm{k}$. For $n\geq 3$ the action of $\GL_n(\mathrm{k})$ on $\mathbb{P}^1(\mathrm{k}^n)$ is not 3-transitive, because 3 linearly dependent vectors cannot be mapped to 3 linearly independent vectors. On the other hand, it is easy to check that this action is 2-transitive and satisfies (flip).

One can construct many other examples of subgroups $G$ of $\Aut(T)$  satisfying the hypotheses of Theorems \ref{thm.measure.classification}, \ref{thm.unique.ergo} using the closure techniques given in \cite{BEW}.

\subsubsection{Contraction groups}

Contraction groups are useful in many situations in ergodic theory and in the study of unitary representations. One instance is the Mautner phenomenon, which we recall below. For a proof, see \cite[Chp. III, Thm. 1.4]{BeMa}.

\begin{lemma}
\label{Mautner}
Let $a \in G$ be a hyperbolic element and let
$(\pi, \mathcal{H})$ be a unitary representation of $G$ on a Hilbert space $\mathcal{H}$. Let $v \in \mathcal{H}$ be a non-zero vector and suppose $\{\pi(a^{n})v\}_{n\geq 1}$ weakly converges to $v_{0} \in \mathcal{H}$. Then $v_{0}$ is $\overline{U}^+$-invariant. \qed
\end{lemma}


Contraction groups for Lie or $p$-adic analytic groups are always closed (see \cite[page 2]{CaMe13}). There is a partial converse to this statement (\cite[Theorem B]{CaMe13}): if $G\leq \Aut(T)$ is closed, non-compact, transitive on $\partial T$ and have torsion-free closed contraction groups, then $G$ is a linear group.%

For groups with Tits' independence property, we can identify the closure of the contraction groups:

\begin{lemma}
\label{lem::closure_contraction_gr}
Let $G \leq \Aut(T)$ be a closed subgroup with Tits' independence property. 
Then the closure of $U^+$ in $G$ is $G^{0}_{\xi}$. 
\end{lemma}

\begin{proof}
Since $G^0_\xi$ is closed and contains $U^+$, it is enough to show that given $g \in G^0_\xi$, one can find a sequence $u_n \in U^+$ converging to $g$. This is seen by Tits' independence property. Since $g \in G^0_\xi$, $g$ fixes pointwise some ray $[x_i,\xi)$. For every $n \geq i$, by Tits' independence property, there exists an element $u_n \in G$, fixing pointwise $T_{[x_{n},x_{n+1}]}$ and agreeing with the action of $g$ on $T_{[x_{n+1},x_{n}]}$. Clearly, $u_n \in U^+$ and $u_n \to g$ as $n \to \infty$, as desired.
\end{proof}

\subsubsection{Haar measures of $G$ and its subgroups}
\label{subsub.haar}
Since $G$ acts $2$-transitively on $\partial T$, $G$ is unimodular (\cite[Proposition 6]{Amann}). 

Explicit descriptions of the Haar measure on $G$ and $G_\xi$ will be used throughout the proofs of Theorems \ref{thm.measure.classification}, \ref{thm.unique.ergo} and  \ref{thm.geo.fin}. These descriptions rely on decompositions from \S \ref{subsub.decompositions} and the following standard fact regarding the product structure of Haar measures.




\begin{lemma}[Product structure of Haar measures]\label{lemma.product.structure}
Let $G$ be a unimodular locally compact Polish group, and $S_1,S_2$ two closed subgroups of $G$ such that the intersection $K:=S_1 \cap S_2$ is compact. Assume that $S_1 S_2$ is $m_G$-conull in $G$. Then,  $m_G$ is proportional to the push-forward of $m_{S_1} \otimes m_{S_2}^r$ under the multiplication map, where $m_{S_2}^r$ denotes the right Haar measure on $S_2$, i.e.
$$\int_{G} f(x)dm_G(x)= \int_{S_1 \times S_2} f(st) \Delta_{S_2}(t) dm_{S_1}(s)dm_{S_2}(t) $$
for any Borel $f \geq 0$ on $G$, where $\Delta_{S_{2}}$ is the modular function of $S_{2}$.
\end{lemma}

\begin{proof}
The subset $\Omega:=S_1S_2 \subseteq G$ is a standard Borel space. The group $S_1 \times S_2$ acts on $\Omega$ and the stabilizer of $\id \in \Omega$ is $diag(K)$ (i.e. $K$ diagonally embedded in $S_1 \times S_2$). Therefore, we have a map $(S_1 \times S_2)/diag(K) \to \Omega$, given by $(s,t)diag(K) \mapsto st^{-1}$. By construction, this is a continuous bijection between two standard Borel spaces, and by Suslin's theorem, it is a Borel isomorphism, meaning that its inverse is also Borel measurable. From here, the proof goes as in  \cite[Proof of Theorem 8.32]{Knapp}. 
\end{proof}




Let us now apply Lemma \ref{lemma.product.structure} to Bruhat and $AN$ decompositions:\\[7pt]
(1) In Lemma \ref{lemma.product.structure}, let $S_1=G^0_{\xi_-}$ and $S_2=G_{\xi}$. The group $S_1 \cap S_2 = G_{(\xi_-,\xi)}$ is compact.  Hence, the restriction of $m_G$ to $G^0_{\xi_-}G_{\xi}$ is proportional to the push-forward of $m_{G^0_{\xi_-}}\otimes m_{G_{\xi}}^r$ by the multiplication map. Note that $G_{\xi}$ does not contain any neighborhood  of identity in $G$, hence $m_G(G_{\xi}^{0})=m_G(G_{\xi})= 0$. Therefore, it follows by Bruhat decomposition that the subset $G^0_{\xi_-} G_{\xi}$ is $m_G$-conull in $G$. \\[-4pt]

\noindent (2) 
The group $G^0_{\xi}$ is an increasing union of the compact groups $G_{[x_k,\xi)}$ and hence, it is unimodular. By $AN$ decomposition $G_{\xi}=a^{\bbZ} G^0_{\xi}$. By Lemma \ref{lemma.product.structure} with $S_1=a^\bbZ$ and $S_2=G^{0}_{\xi}$, $m_{G_{\xi}}$ is proportional to the push-forward of $m_{S_1} \otimes m^r_{G^0_{\xi}}$, where $m_{S_1}$ is the counting measure.

We normalize $m_{G^0_{\xi}}$ so that $m_{G^0_{\xi}}(G_{[x_0,\xi)})=1$ and similarly for $m_{G^0_{\xi_-}}$. Since $G$ acts $2$-transitively on $\partial T$, $G_{\xi}$ acts transitively on $\partial T \setminus \{\xi\}$. Hence for $k \in \mathbb{N}$, and $\ell$ an even integer, $G_{[x_{\ell},\xi)}$ has index $(d_0-1)^{\lfloor \frac{k}{2} \rfloor}(d_1-1)^{\lceil \frac{k}{2}  \rceil}$ in $G_{[x_{\ell+k},\xi)}$ (if $\ell$ is odd, to find the index, exchange $d_0$ and $d_1$). Consequently, for $k \in \mathbb{Z}$,  $m_{G^0_{\xi}}(G_{[x_k,\xi)})=(d_0-1)^{\lfloor \frac{k}{2} \rfloor}(d_1-1)^{\lceil \frac{k}{2}  \rceil}$.

From the above descriptions of $m_G$ and $m_{G_{\xi}}$, we deduce the following formula describing the restriction of $m_G$ to $G^0_{\xi_-} G^0_{\xi}$.

\begin{cor}\label{lemma.Haar.VAU} For any $f \in C_c(G)$ supported on $G^0_{\xi_-} G^0_{\xi}$, we have 
$$
\int_G f(g)dm_G(g)= \int_{G^0_{\xi}} \int_{G^0_{\xi_-}}  f(v u) dm_{G^0_{\xi_-}}(v) dm_{G^0_{\xi}}(u).
$$ 
\end{cor}

\subsection{Ergodic theory for amenable groups}

In the transverse divergence argument in the proof Theorem \ref{thm.measure.classification}, we need to replace the use of Birkhoff's ergodic theorem in the study of one parameter flows, by a pointwise ergodic theorem for amenable groups, in our case $G^0_{\xi}$. Such a theorem holds with averaging over well-behaved sets. These sets come from tempered F{\o}lner sequences introduced by Shulman.

\subsubsection{Pointwise ergodic theorem for the amenable group $G_{\xi}^{0}$}
\label{sec::ergodic_theorem}

A sequence $F_1, F_2, \dots$ of compact subsets of a locally compact group $H$ is called a \textbf{F{\o}lner sequence} if for every compact $K \subseteq H$ and $\varepsilon>0$, there exists $N>0$ such that $m_H(F_i \; \Delta \;  KF_i) < \varepsilon  m_H(F_i)$, for all $i \geq N$. 
The sequence $F_i$ is called a \textbf{tempered F{\o}lner sequence}  if in addition there is $C >0$ such that for every $i \in \bbN$ 
$$m_{H}\left(\bigcup\limits_{k=1}^{i-1}F^{-1}_{k} F_i\right) \leq C   m_{H}(F_i).$$

The following result is due to Lindenstrauss.

\begin{theorem}(Pointwise ergodic theorem for amenable groups \cite[Theorem 1.2]{Lind})
\label{thm.lindenstrauss.ergodic}
Let $H$ be amenable group acting ergodically by measure preserving transformations on a Lebesgue probability space $(X,\mu)$. Let $F_i$ be a tempered F{\o}lner sequence of $H$. Then, for every $f \in L^1(X)$, for $\mu$-a.e. $x \in X$
\begin{equation}\label{eq.erg.convergence}
\frac{1}{m_H(F_i)}\int_{F_i} f(g x) dm_H(g) \underset{i \to \infty}{\longrightarrow} \int_{X} f d\mu .
\end{equation} 

\end{theorem}

The following lemma will allow us to construct many tempered F{\o}lner sequences for $G_\xi^0$. 

\begin{lemma}
\label{lem::general_Folner} 
With the notation of \S \ref{subsection.notation}, let $F \subset G^0_\xi$ be an $M$-invariant compact subset of $G_\xi^0$ of positive $m_{G^0_\xi}$ measure. Then the sequence $a^i F a^{-i}$ is a tempered F{\o}lner sequence for $G^0_\xi$.
\end{lemma}

\begin{remark}[Contraction to identity]\label{remark.folner.if}
We note that a similar observation in linear algebraic setting appears in \cite[Lemma 7.2]{margulis-tomanov}. The difference in the hypotheses is due to the fact that in linear algebraic case, the conjugation by $a^{-1}$ contracts any compact subset of a unipotent group to identity, whereas in our case this conjugation contracts towards the subgroup $M$ of $G^0_\xi$.
\end{remark}

\begin{proof}[Proof of Lemma \ref{lem::general_Folner}]
Let $F$ be as in the statement. Let $K$ be a compact subset of $G^0_\xi$. By compactness, there exists $j \in \bbZ$ such that $K \subseteq G_{[x_j,\xi)}$. Futhermore, for any neighborhood $U$ of identity in $G^0_\xi$, there exists $N_{U}\geq 1$ such that $a^{-i}Ka^i \subseteq UM$ for $i \geq N_{U}$. Indeed, otherwise one can find a sequence of elements $g_{i_\ell} \in a^{-i_\ell}Ka^{i_\ell} \subseteq G_{[x_{j-i_\ell},\xi)}$ and $g_{i_\ell} \notin UM$. For any $k\leq 0$ the sequence $g_{i_\ell}$ eventually belongs to $G_{[x_k,\xi)}$, hence some limit point $g$ of $g_{i_\ell}$ is in $M$, a  contradiction. 

Since $a$-action by right multiplication multiplies the measure by the value of the modular function of $G_\xi$, we have



\begin{equation}\label{eq.folner.lemma}
\frac{m_{G^0_\xi}((Ka^{i}Fa^{-i}) \Delta (a^i F a^{-i}))}{m_{G^0_\xi}(a^i F a^{-i})}=\frac{m_{G^0_\xi}((a^{-i}Ka^{i}F) \Delta F)}{m_{G^0_\xi}(F)}. 
\end{equation}
By above, for all $i$ large enough, we have
$$ a^{-i}Ka^iF \setminus F \subset UMF \setminus F \subset UF \setminus F$$
and the measure of the latter converges to $0$ when $U$ shrinks to identity. Similarly, as $a^{-i}Ka^i$ shrinks to $M$, for all large $i$, $a^{-i}Ka^i$ contains an element $u_im$, with $m\in M$ and $u_i$ converging to identity. Hence,
$$ F \setminus a^{-i}Ka^iF  \subset F \setminus u_imF \subset F \setminus u_iF $$
and the measure of the latter set converges to $0$ as $i \to +\infty$ (see e.g. \cite[Chapter XII, Theorem A]{halmos}).
Combining last two observations, the expression in \eqref{eq.folner.lemma} goes to $0$ as $i\to \infty$, showing that $a^{i}F a^{-i}$ is a F{\o}lner sequence.

To see that it is also tempered:
$$
\frac{m_{G^0_\xi}\left(\bigcup\limits_{k=1}^{i-1} a^k F^{-1} a^{-k} a^{i} F a^{-i}\right)}{m_{G^0_\xi}(a^i F a^{-i})}=\frac{m_{G^0_\xi}\left(\bigcup\limits_{k=1}^{i-1} a^{k-i} F^{-1} a^{i-k} F\right)}{m_{G^0_\xi}(F)} < \infty.
$$
The last inequality follows, since by compactness of $F$ there exists $j' \in \bbZ$ such that $ F$ and all of its $a^{k-i}$ conjugates are contained in $ G_{[x_{j'},\xi)}$.
\end{proof}

The tempered F{\o}lner sequence described in the following example will be extensively used in Sections \ref{sec.main.proof} and \ref{sec.proof.of.prop}.

\begin{example}[A tempered F{\o}lner sequence]\label{ex.F_i}
For $i \in \mathbb{Z}$, set $F_i:=G_{[x_i,\xi)} \setminus G_{[x_{i-1},\xi)}$. It is readily seen that for each even $i \in \mathbb{Z}$, $F_i$ is a union of $(d_0-2)$ ($(d_1-2)$ if $i$ is odd) non-trivial cosets of $G_{[x_{i-1},\xi)}$ in $G_{[x_i,\xi)}$, in particular, it is compact, $M$-invariant and of positive measure.  
Moreover, for $i,n \in \mathbb{Z}$, we have $a^n F_i a^{-n}=F_{i+2n}$. 
Thus, applying Lemma \ref{lem::general_Folner} to $F_0$ and $F_1$, we deduce that $F_{2n}$ and $F_{2n+1}$ are tempered F{\o}lner sequences for $G^0_\xi$. 
\end{example}

\subsubsection{$\mu$-Uniformly generic  sets}\label{subsub.unif.gen}
Let $H$ be an amenable group acting continuously on a locally compact second countable topological space $X$. Let $\mu$ be an $H$-invariant ergodic probability measure on $X$ and  $F_i$ a tempered F{\o}lner sequence for $H$. By Theorem \ref{thm.lindenstrauss.ergodic}, for a fixed function $f \in C_c(X)$ there exists a $\mu$-full measure subset $X_f \subseteq X$ such that the convergence $(\ref{eq.erg.convergence})$ holds for each $x \in X_f$. 

A point $x \in X$ is called \textbf{$\mu$-generic} if the convergence $(\ref{eq.erg.convergence})$ holds for $\textit{every}$ $f \in C_c(X)$. Since $C_c(X)$ is separable, taking the intersection of $X_{f_i}$'s over a countable dense subset $\{f_i\}$ of $C_c(X)$, there exists a $\mu$-full measure subset $X_{\mu} \subseteq X$ consisting of $\mu$-generic points. 

Since the speed of convergence in $(\ref{eq.erg.convergence})$ may depend on $x \in X_{\mu}$, we will need the following strengthening of a set of $\mu$-generic points (see \S 7.3 in \cite{margulis-tomanov}) over a finite collection of tempered F{\o}lner sequences.

\begin{lemma}[Uniformly generic sets]\label{lemma.unif.gen}
Let $H$ be a locally compact 
amenable group acting continuously on a locally compact second countable topological space $X$ and let $\mu$ be an $H$-invariant and ergodic probability measure on $X$. Given a finite collection $\mathcal{C}$ of tempered F{\o}lner sequences for $H$ and $\delta>0$, there exists a Borel subset $K \subset X$ with $\mu(K)>1-\delta$ and such that for every $f \in C_c(X)$ and $\varepsilon>0$ there exists $n_{f,\varepsilon} \in \mathbb{N}$  satisfying 
$$
\left|  \frac{1}{m_H(F_i)} \int_{F_i} f(g x) dm_H(g) - \int_X f(x) d\mu(x) \right| <\varepsilon
$$
for every $i \geq n_{f,\varepsilon}$, $x \in K$ and any sequence $F_i$ in $\mathcal{C}$. 
\end{lemma}

We call a subset $K \subset X$ satisfying the conclusion of Lemma \ref{lemma.unif.gen} a \textbf{$\mu$-uniformly generic set} for the finite collection of tempered F{\o}lner sequences $\mathcal{C}$. 

\begin{proof}
Noting that $C_c(X)$ is separable, the proof follows from Theorem \ref{thm.lindenstrauss.ergodic} and Egoroff's theorem.
\end{proof}


Before proceeding further, we also single out the following standard application of Fubini's theorem for reader's convenience. It will be used on several occasions. 


\begin{lemma}[Invariant measures and Fubini]\label{lemma.fubini}
Let $G$ be a locally compact group endowed with a $\sigma$-finite measure $m$, acting measurably on a $\sigma$-finite measure space $(X,\mu)$ and preserving the measure $\mu$. Let $Y$ be a measurable  $\mu$-conull subset of $X$. Then for $m$-a.e. $g \in G$ and $\mu$-a.e. $x\in X$ we have $g x \in Y$. \qed
\end{lemma}



\subsection{Differentiation of measures on $G_\xi^0$}
Here, our goal is to obtain a differentiation theorem for $G_\xi^0$ in order to  compute the Radon--Nikodym derivatives locally. We use a classical result from geometric measure theory (see Federer \cite[Sections 2.8--2.9]{federer}). To proceed, by the result of Haagerup and Przybyszewska \cite{Haagerup}, 
endow $G^0_\xi$ with a compatible proper left-invariant metric. We note in passing that in our case it is not hard to construct such a metric explicitly.
Each $F_i$ admits a natural metric from its action on a rooted tree and one can take a box space of the sets $F_i$ (see e.g. \cite{khukro} for definition of box spaces). 
 We omit the details of the construction. To verify that differentiation theory applies in our situation, we note the following.

\begin{lemma}
For integers $i \geq 1$ and $j \in \mathbb{Z}$, let $P_{i,j} :=  \{  gG_{B(x_j,i),\xi} \; \vert \; g \in G^0_\xi\}$. For every $i,j$, $P_{i,j}$ is a Borel partition of $G_\xi^0 $, every member of $P_{i,j}$ is bounded and is a union of finitely many members of $P_{i+1,j}$. Additionally, for every $j \in \mathbb{Z}$,
\[  \limsup_{i\to \infty} \{ \diam(S)\;  \vert \;  S\in P_{i,j} \} = 0 .\]
\end{lemma}
\begin{proof}
Note that $G_{B(x_j,i),\xi}$'s constitute a basis of compact open neighborhoods of identity in $G^0_\xi$ and each $P_{i,j}$ is a collection of cosets of $G_{B(x_j,i),\xi}$. Therefore, $P_{i,j}$'s are Borel partitions. Each member of $P_{i,j}$ is bounded since the metric is left-invariant, proper and compatible with the topology of $G^0_\xi$. The third claim follows since $G_{B(x_j,i+1),\xi}$ is a finite index subgroup of $G_{B(x_j,i),\xi}$. Finally, the last claim follows directly by compatibility and left-invariance of the metric on $G^0_\xi$.
\end{proof}

\begin{prop}\label{prop.differentiation_of_measures}
Let $\Phi: G_\xi^0 \to G_\xi^0$ be a measurable map such that the pushforward measure $\Phi_*m_{G_\xi^0}$ is absolutely continuous with respect to $m_{G_\xi^0}$.  Then, for each $j \geq 1$, and for $m_{G_\xi^0}$-almost every $g\in G^0_\xi$, we have
\[ \lim_{r\to \infty} \frac{m_{G_\xi^0}(\Phi^{-1}(gG_{B(x_j,r),\xi}))}{m_{G_\xi^0}(gG_{B(x_j,r),\xi})} =\frac{d\Phi_*m_{G_\xi^0}}{dm_{G_\xi^0}} (g). \]
\end{prop}
\begin{proof}
By \cite[Theorem 2.8.19]{federer}, $V_j=\{(x,S)  \; \vert \; x\in S\in P_{i,j} \text{ for some }i\in \bbN\}$ is a Vitali relation (see \cite[Section 2.8.16]{federer} for definitions). Then the conclusion of the proposition is just a reformulation in our setting of \cite[Theorem 2.9.7 and Section 2.9]{federer}.
\end{proof}

\section{Unique ergodicity of compact quotients}\label{sec.unique.ergo}

The goal of this section is to prove Theorem \ref{thm.unique.ergo}. It will follow from the following special case.

\begin{proposition}\label{prop.hyp.unique.ergo}
Keep the hypotheses and the notation of Theorem \ref{thm.unique.ergo}. Suppose furthermore that $\Gamma$ is  torsion-free. Then the $G^0_\xi$-action on $X$ is uniquely ergodic. 
\end{proposition}

\begin{proof}
Since $G_{[x_0,\xi)}$ is a compact, left-$M$-invariant set of positive measure, by Lemma \ref{lem::general_Folner}  $G_{[x_{2n},\xi)}=a^n G_{[x_0,\xi)} a^{-n}$ is a tempered F{\o}lner sequence for $G^0_{\xi}$. Our goal is to show that for every $\theta \in C(X)$, we have the following convergence uniformly in $x \in X$:
\begin{equation}\label{eq.unique.ergo0}
\frac{1}{m_{G^0_{\xi}(G_{[x_{2n},\xi)})}} \int_{G_{[x_{2n},\xi)}} \theta(u x) dm_{G^0_{\xi}}(u) \underset{n \to \infty}{\longrightarrow} \int_X \theta(y)dm_X(y).
\end{equation}

Let $\varepsilon >0$, $\theta \in C(X)$ and $x \in X$. We define $\tilde{\theta} : X \to \bbR$ by 
\[ \tilde{\theta}(z):=\int_{M} \theta(m z)dm_M(m),\] where $m_M$ is the Haar measure of the compact group $M= G^0_\xi \cap G^0_{\xi_-}$ with normalization $m_M(M)=1$. The map $\tilde{\theta}$ is $M$-invariant and, by compactness of $X$, is uniformly continuous. 

Now we argue that there exists $k>0$ such that for every $w \in G_{[x_{-2k},x_{2k}]}$ and $z \in X$:
\begin{equation}\label{eq.unique.ergo1}
|\tilde{\theta}(w z)-\tilde{\theta}(z)|<\varepsilon.
\end{equation}
Indeed, since $\tilde{\theta}$ is uniformly continuous and $G$ acts continuously on $X$, inequality \eqref{eq.unique.ergo1} is true for every $w$ belonging to a sufficiently small neighborhood $O$ in $G$ of the identity element. By Lemma \ref{lem::some_little_facts}.(3) below, for sufficiently large $k$, for every $w \in G_{[x_{-2k},x_{2k}]}$ we can find $m \in M$ such that $mw \in O$. Consequently, since $\tilde{\theta}$ is $M$-invariant, $(\ref{eq.unique.ergo1})$ holds for each $ w \in G_{[x_{-2k},x_{2k}]}$. 

Fix $k >0 $ such that \eqref{eq.unique.ergo1} holds. For $z,y \in X$ we set $$\phi_z(y):=\frac{1}{m_X(G_{[x_{-2k},x_{2k}]} z)}\mathbbm{1}_{G_{[x_{-2k},x_{2k}]} z}(y).$$  The map $z \mapsto \phi_z$ is locally constant as a function  $X \to L^2(X,m_X)$, because $G_{[x_{-2k},x_{2k}]} z$ is compact and open in $X$. In particular, $z \mapsto \phi_z$ is continuous, and $\{\phi_z\}_{z\in X}$ is a compact subset of $L^2(X,m_X)$.

Since $G$ has the Howe--Moore property (see \cites{BM00a,Cio}), the corresponding action of $a$ on $L^2(X,m_{X})$ is  mixing, therefore we have
\begin{equation}\label{eq.unique.ergo22}
\int_X \phi_z(a^{-\ell} y) \tilde{\theta}(y)dm_X(y) \underset{\ell \to \infty}{\longrightarrow} \int_X \phi_z(y)dm_X(y)   \int_X \tilde{\theta}(y)dm_X(y).
\end{equation}
Moreover, since $\{\phi_z \}_{z\in X}$ is compact, this convergence is uniform in $z \in X$.
Hence we can choose $\ell$ large enough so that the right-hand side of $(\ref{eq.unique.ergo22})$ is within $\varepsilon$ of the left-hand side for every $z \in X$. Now we treat separately the right and left-hand sides of \eqref{eq.unique.ergo22} ultimately relating them to the right and left-hand sides of \eqref{eq.unique.ergo0}.

\textit{Right-hand side of \eqref{eq.unique.ergo22}}: By construction $\int_X \phi_z(y)dm_X(y)=1$, for every $z \in X$. Since $m_X$ is $G$-invariant, by Fubini's theorem, we have
\begin{equation}\label{eq.unique.ergo3}
\int_X \tilde{\theta}(y) dm_X(y)=\int_X \theta(y) dm_X(y).
\end{equation}


\textit{Left-hand side of \eqref{eq.unique.ergo22}}:
We have
\begin{equation*}
\begin{split}
\int_X \phi_z(a^{-\ell} y) \tilde{\theta}(y)dm_X(y)&=\int_X \phi_z(y) \tilde{\theta}(a^\ell y) dm_X(y)\\
&= \frac{1}{m_X(G_{[x_{-2k},x_{2k}]} z)} \int_{G_{[x_{-2k},x_{2k}]} z} \tilde{\theta}(a^\ell y) dm_X(y)\\
\end{split}
\end{equation*}
\begin{equation}
\label{eq.unique.ergo5}
\begin{split}
&=\frac{1}{m_G(G_{[x_{-2k},x_{2k}]})} \int_{G_{[x_{-2k},x_{2k}]}} \tilde{\theta}(a^\ell g z) dm_G(g)\\
&= (d_0-1)^{2k}(d_1-1)^{2k} \int_{a^{\ell} G_{[x_{-2k},x_{2k}]} a^{-\ell}} \tilde{\theta}(g a^\ell z)dm_G(g)\\
&= (d_0-1)^{2k}(d_1-1)^{2k} \int_{G_{[x_{-2(k-\ell)},\xi)}} \int_{G_{[x_{2(k+\ell)},\xi_-)}} \tilde{\theta}(vua^\ell z) dm_{G^0_{\xi_-}}(v) dm_{G^0_{\xi}}(u).\\
\end{split}
\end{equation}

The first equality is by $G$-invariance of $m_X$ and the second one by definition of $\phi_z$. The third one follows from injectivity of the maps $w \mapsto w z$ on $G_{[x_{-2k},x_{2k}]}$ for every $z \in X$ (see Lemma \ref{lem::some_little_facts} (2)), and thus we can lift this integral from $X$ to $G$. The fourth one follows by the unimodularity of $G$, 
and the fact that $m_G(G_{[x_{-2k},x_{2k}]})=(d_0-1)^{-2k}(d_1-1)^{-2k}$ (which is readily verified using (1) of Lemma \ref{lem::some_little_facts} and the product structure). The last equality in $(\ref{eq.unique.ergo5})$ follows from Lemma \ref{lemma.Haar.VAU}, since by Lemma \ref{lem::some_little_facts} (1), we have
$$a^{\ell} G_{[x_{-2k},x_{2k}]} a^{-\ell}=a^\ell  G_{[x_{2k},\xi_-)}a^{-\ell} a^\ell  G_{[x_{-2k},\xi)} a^{-\ell}=G_{[x_{2(k+\ell)},\xi_-)} G_{[x_{-2(k-\ell)},\xi)}.$$

As $G_{[x_{2(k+\ell)},\xi_-)}\subset G_{[x_{-2k},x_{2k}]}$,  $(\ref{eq.unique.ergo1})$ implies that for every $z \in X$ and for every $v\in G_{[x_{2(k+\ell)},\xi_-)}$, we have $|\tilde{\theta}(vua^\ell z)-\tilde{\theta}(u a^\ell z)|<\varepsilon$. Since $m_{G^0_{\xi_-}}(G_{[x_{2(k+\ell)},\xi_-)})=(d_0-1)^{-k-\ell}(d_1-1)^{-k-\ell}$, it follows that for every $z\in X$, the last quantity in $(\ref{eq.unique.ergo5})$ is $\varepsilon$-close to
\begin{equation*}
\begin{split}
& (d_0-1)^{k-\ell}(d_1-1)^{k-\ell}   \int_{G_{[x_{2(\ell-k)},\xi)}} \tilde{\theta}(ua^\ell z)dm_{G^0_{\xi}}(u)\\=
&\frac{1}{m_{G^0_{\xi}}(G_{[x_{2(\ell-k)},\xi)})}   \int_{G_{[x_{2(\ell-k)},\xi)}} \theta(ua^\ell z)dm_{G^0_{\xi}}(u).\\
\end{split}
\end{equation*}
The last equality follows from the definition of $\tilde{\theta}$ and Fubini's theorem. 

Choosing $z=a^{-\ell}x$, we have that the last expression is equal to the left-hand side of $\eqref{eq.unique.ergo0}$ with $n=l-k$, and it is within $2\varepsilon$ of the right-hand side of $\eqref{eq.unique.ergo0}$, concluding the proof.
\end{proof}

In the proof of Proposition \ref{prop.hyp.unique.ergo} we have used the following lemma.

\begin{lemma}
\label{lem::some_little_facts}
Let $k \geq 1$. Under the hypotheses of Theorem \ref{thm.unique.ergo}, we have:
\begin{enumerate}
    \item $G_{[x_{-2k},x_{2k}]}=G_{[ x_{2k},\xi_-)}   G_{[x_{-2k},\xi)}$,
    \item If $\Gamma$ is torsion-free, then for every $x \in X=G/\Gamma$ the map $w \mapsto w x$ is injective on $G_{[x_{-2k},x_{2k}]}$,
    \item for each $w \in G_{[x_{-2k},x_{2k}]}$, there exists $m \in M$ such that $m w \in G_{B(x_0, 2k-1)}$.
\end{enumerate}
\end{lemma}

\begin{proof}
In $(1)$, the inclusion $\supseteq$ is clear, the other inclusion follows by 2-transitivity. 
For (2), suppose that for some $x \in X$, $x=g\Gamma$, the map is not injective. Then, there exists $w \neq \id \in G_{[x_{-2k},x_{2k}]}$ such that the non-trivial elliptic element $g^{-1}wg$ belongs to $\Gamma$. But since $\Gamma$ is torsion-free, all its non-trivial elements are hyperbolic, yielding a contradiction.


Finally, let us prove (3). Let $w \in G_{[x_{-2k},x_{2k}]}$. By (1), write $w= w_1 w_2$, with $w_1 \in G_{[ x_{2k},\xi_-)}$, and $w_2 \in G_{[x_{-2k},\xi)}$. By Tits' independence property, we can write $w_1=m_1w'_1$, with $w'_1$ fixing the half tree $T_{[x_{2k},x_{2k-1}]}$ and $m_1$ fixing the half tree $T_{[x_{2k-1},x_{2k}]}$. In particular, $m_1 \in M$ and $w_1' \in G_{B(x_0,2k-1)}$. Similarly one can write $w_2=m_2w'_2$ with $m_2 \in M$ and $w_2' \in G_{B(x_0,2k-1)}$. Moreover, $m_2^{-1}w_1'm_2 \in G_{B(x_0,2k-1)}$. Now one can write $w=m_1m_2 (m_2^{-1}w_1'm_2)w_2'$, and the claim follows with $m=m_2^{-1}m_1^{-1}$.
\end{proof}
 
We are now in a position to prove Theorem \ref{thm.unique.ergo}, using Proposition \ref{prop.hyp.unique.ergo}:

\begin{proof}[Proof of Theorem \ref{thm.unique.ergo}]
Since $\Gamma$ is a uniform lattice in $G$, by \cite[Prop 7.9.(a)]{Bass} it is finitely generated and by \cite[Theorem 8.3.(c)]{Bass}, it contains a finite index free subgroup $\Gamma_0$, in particular, $\Gamma_0$ is torsion-free. Let $X_0$ be the compact quotient $G/\Gamma_0$. The natural projection $X_0 \to X$ given by $g\Gamma_0 \mapsto g \Gamma$ is $G$-equivariant. Therefore, $X$ is a factor of $X_0$.
By Proposition \ref{prop.hyp.unique.ergo}, $G^0_\xi$-action on $X_0$ is uniquely ergodic. Since $X$ is a factor of $X_0$, by a standard argument, $G^0_\xi$-action on $X$ is also uniquely ergodic and the theorem is proven. 
\end{proof}


\section{Proof of Theorem \ref{thm.measure.classification}}\label{sec.main.proof}


\subsection{Before the divergence}\label{subsec.before.divergence}

Let $\mu$ be a $G^0_{\xi}$-invariant and ergodic Borel probability measure on $X$. 

\subsubsection{$G_{\xi}^0$-homogeneous case}
\label{sec::homogeneous_case}

First, suppose that there exists $x \in X$ such that $\mu(G_{\xi}^0 x)>0$. By $G_{\xi}^0$-ergodicity of $\mu$, we deduce $\mu(G_{\xi}^0 x)=1$. Let $x=g\Gamma$. The $G_{\xi}^0$-equivariant Borel isomorphism $$G_{\xi}^0 /(g\Gamma g^{-1} \cap G_{\xi}^0) \cong G_{\xi}^0 x \subset X$$ induces a $G_{\xi}^0$-invariant probability measure on $G_{\xi}^0 /(g\Gamma g^{-1} \cap G_{\xi}^0)$. In particular, $g\Gamma g^{-1} \cap G_{\xi}^0$ is a lattice in $G_{\xi}^0$, and therefore $G_{\xi}^0 x$ is closed in $X$ (see e.g. \cite[Theorem 1.13]{raghunathan}). As a conclusion, $\mu$ is the  $G_{\xi}^0$-homogeneous probability measure on $X$ supported on the closed orbit $G_{\xi}^0 x$.

\subsubsection{A shrinking sequence of differences}
In view of the previous observation, we suppose for the rest of the section that for every $x \in X$ we have $\mu(G_{\xi}^0 x)=0$. With this assumption and using Lemma \ref{lemma.unif.gen}, the next lemma will allow us to find arbitrarily close pairs of points in a set of uniformly generic points differing by an element outside of $G_{\xi}$.

\begin{lemma}\label{lemma.elliptics.to.identity}
Let $Y$ be a Borel subset of $X$ such that $\mu(Y)>0$. Then there exist $y \in Y$ and a sequence $\epsilon_n \subset G \setminus G_{\xi}^0$ such that $\epsilon_n \xrightarrow{n \to \infty} \id$ and $\epsilon_n  y \in Y$ for every $n \in \mathbb{N}$.
\end{lemma}

\begin{proof}
By regularity of $\mu$ and up to replacing $Y$ with a compact set $K \subseteq Y$ with $\mu(K)>0$, we can suppose without loss of generality that $Y$ is compact. 

For a contradiction, we assume the conclusion of the lemma is not true. Then every $y \in Y$ 
admits an open neighborhood $V \subset G$ of $\id$ such that $V y \cap Y \subseteq G_{\xi}^0 y$. 
As the projection $\pi_y:G \to X$ defined by $ g \mapsto \pi_y(g):=g y$ is an open map, $V y$ is open in $X$. By compactness of $Y$, there exist $y_1,\ldots,y_n \in Y$ such that $Y \subseteq \bigcup\limits_{i=1}^n G_{\xi}^0 y_i$. In particular, there exists $i \in \{1,\ldots,n\}$ with $\mu(G_{\xi}^0 y_i)>0$. This contradicts the hypothesis that $\mu(G_{\xi}^0 x)=0$ for every $x \in X$ and the lemma is proven.
\end{proof}

\subsubsection{A useful factorization and associated change of variables for time shift}
The following key technical proposition will allow us, with Lemma \ref{lemma.elliptics.to.identity}, to carry out the divergence argument with ``different speeds", similarly to the one used in the study of algebraic flows by Ratner \cite{ratner.unipotent} (see also Margulis-Tomanov \cite{margulis-tomanov}). Proposition \ref{prop.variable.change.RN} is proven in \S \ref{subsec::proof_change.RN}.
To state it, we describe a finite collection of tempered F{\o}lner sequences for $G_\xi^0$. It will become clear in the proof that this is a convenient choice. 

We choose a finite collection of elements $C_{\text{even}} \subset G_{[x_{-1},\xi)}$ such that for every $g\in G_{[x_{-1},\xi)}$ there exists $t\in C_{\text{even}}$ such that $g|_{B(x_{-1},2)}=t|_{B(x_{-1},2)}$. Similarly, choose a finite set $C_{\text{odd}} \subset G_{[x_0,\xi)}$ so that for every $g\in G_{[x_{0},\xi)}$ there exists $t\in C_{\text{odd}}$ such that $g|_{B(x_{0},2)}=t|_{B(x_{0},2)}$. Recall that $F_n = G_{[x_n,\xi)} \setminus G_{[x_{n-1},\xi)}$ (see Example \ref{ex.F_i}) and let $F_n^t : = a^n F_0 t a^{-n}$ for $t\in C_{\text{even}}$ and $F_n^{t'}:= a^n F_1 t' a^{-n}$ for $t'\in C_{\text{odd}}$. If $t,t'=\id_{G}$, these correspond to the sequences $F_{2n}$ and $F_{2n+1}$.

Denote by 
\begin{equation}\label{def.folner.collection}
    \mathcal{C}:= \left\{  \{F_n^t\}_{n\in \bbN} \; | \;    t\in C_{\text{even}}  \right\} \cup \left\{  \{F_n^{t'}\}_{n\in \bbN} \; | \;    t'\in C_{\text{odd}}  \right\}.
\end{equation}
Observe that for each $t\in C_{\text{even}}$ and $t'\in C_{\text{odd}}$ the sets $F_0 t$ and $F_1t'$ are compact, left-$M$-invariant and have positive $m_{G_\xi^0}$-measures. Hence, by Lemma \ref{lem::general_Folner}, all members in $\mathcal{C}$ are tempered Folner sequences for $G_\xi^0$. 

\begin{proposition}\label{prop.variable.change.RN}
Let $\mathcal{C}$ be the finite collection of tempered F{\o}lner sequences for $G^0_{\xi}$ as above. For any sequence of elements $\epsilon_n \in G \setminus G^0_{\xi}$ with $\epsilon_n \xrightarrow{n \to \infty} \id$, up to possibly reindexing $\epsilon_n$ and passing to subsequence, for every $n\in \bbN$ there exist a continuous map $\Psi_{\epsilon_n}: F_{n} \to G$
and a continuous injection $\Phi_{\epsilon_n}: F_{n} \to  G$ such that
 \begin{enumerate}
     \item \label{prop.variable.change.RN_1}
     $u\epsilon_{n}=\Psi_{\epsilon_{n}}(u)a \Phi_{\epsilon_{n}}(u)$, for each $u \in F_{n}$,
     \item \label{prop.variable.change.RN_2}
     for every neighborhood $O$ of identity in $G$, and for all $n$ large enough, we have $\Psi_{\epsilon_{n}}(F_{n}) \subseteq O$,
     \item \label{prop.variable.change.RN_3}
     $\Phi_{\epsilon_n}(F_n)$ is a subsequence of a tempered F{\o}lner sequence in $\mathcal{C}$,
     \item \label{prop.variable.change.RN_4}
     for each $n \in \bbN$, the Radon--Nikodym derivative induced by  $\Phi_{\epsilon_n}: F_{n} \to  \Phi_{\epsilon_n}(F_{n})$  satisfies $$\frac{d(\Phi_{\epsilon_n})_*m_{G^0_{\xi}}}{dm_{G^0_{\xi}}} \vert_{\phi_{\epsilon_n}(F_n)} \equiv
     \frac{m_{G^0_{\xi}}(F_{n})}{m_{G^0_{\xi}}(\Phi_{\epsilon_n}(F_{n}))}.$$ 
 \end{enumerate} 
\end{proposition}

\subsection{The proof}
We remind that for the rest of the proof, our hypothesis is that for every $x \in X$ we have $\mu(G_{\xi}^0 x)=0$. 

\subsubsection{Geometric transverse divergence and $a$-invariance of $\mu$}
\label{sec::transverse_divergence}

The following proposition is the main step of the proof. We obtain additional invariance of $\mu$ under the hyperbolic element $a$ by employing a geometric analogue of the shearing argument of Ratner (see e.g. \cites{ratner.unipotent,ratner.padic,ratner.raghunathan,margulis-tomanov} and the useful notes \cites{quint.notes,benoist.notes}).

\begin{proposition}
\label{prop::a_invariance}
The probability measure $\mu$ is $a$-invariant and, in particular, $G_{\xi}$-invariant.
\end{proposition}

\begin{proof}
Let $f \in C_c(X)$. We denote by $f^a$ the function $f^a(x):=f(a x)$. Let $\mathcal{C}$ be the finite collection of tempered F{\o}lner sequences for $G^0_{\xi}$, as defined before Proposition \ref{prop.variable.change.RN}. By Lemma \ref{lemma.unif.gen}, fix a $\mu$-uniformly generic subset $K \subset X$ of positive $\mu$-measure for $\mathcal{C}$. Then, by Lemma \ref{lemma.elliptics.to.identity} 
, there exist a point $y_0 \in K$ and a sequence $\epsilon_n \in G \setminus G_{\xi}^{0}$ such that $\epsilon_n \xrightarrow{n \to \infty} \id$ and $\epsilon_n y_0 \in K$, for every $n \geq 1$. 
By Proposition \ref{prop.variable.change.RN}, 
up to passing to subsequence and reindexing, one obtains the maps $\Phi_{\epsilon_n}, \Psi_{\epsilon_n}$. 


Let $\delta>0$. By the continuity of the $G^0_{\xi}$-action on $X$, uniform continuity of $f$ and Proposition \ref{prop.variable.change.RN}$(2)$, for large enough $n$ and any $ u \in F_n, x \in X$ 
\begin{equation}\label{eq.drift1}
\left| f(x)-f(\Psi_{\epsilon_n}(u) x) \right| < \frac{\delta}{4}.
\end{equation}

We apply pointwise ergodic theorem for the amenable group $G_\xi^0$. By choice of $K$ from Lemma \ref{lemma.unif.gen} and since $\epsilon_n y_0 \in K$ for every $n\in \bbN$, we have for sufficiently large $n$
\begin{equation}\label{eq.drift2}
\left| \frac{1}{m_{G^0_{\xi}}(F_n)} \int_{F_n} f(u \epsilon_{n} y_0) dm_{G^0_{\xi}}(u) - \int_X f(x) d\mu(x) \right| < \frac{\delta}{4}.
\end{equation} 

By Proposition \ref{prop.variable.change.RN}$(1)$, combining  $(\ref{eq.drift1})$ and $(\ref{eq.drift2})$, for large enough $n$

\begin{equation}\label{eq.drift3}
\left| \frac{1}{m_{G^0_{\xi}}(F_n)} \int_{F_n} f^a(\Phi_{\epsilon_n}(u) y_0) dm_{G^0_{\xi}}(u) - \int_X f(x) d\mu(x) \right| <\frac{\delta}{2}.
\end{equation}

Since, by Proposition \ref{prop.variable.change.RN}
the map $\Phi_{\epsilon_{n}} : F_{n} \to \Phi_{\epsilon_{n}}(F_{n})$ is a continuous bijection, for the first integral in \eqref{eq.drift3}, the following equality holds
\begin{equation*}
\frac{1}{m_{G^0_{\xi}}(F_n)} \int_{F_{n}} f^a(\Phi_{\epsilon_{n}}(u) y_0) dm_{G^0_{\xi}}(u)=\frac{1}{m_{G^0_{\xi}}(F_n)} \int_{\Phi_{\epsilon_{n}}(F_{n})} f^a(u y_0) d(\Phi_{\epsilon_{n}})_{\ast} m_{G^0_{\xi}}(u).
\end{equation*}

By Proposition \ref{prop.variable.change.RN}$(4)$, we have


\begin{equation}
\label{eq.drift_5}
\begin{split}
& \frac{1}{m_{G^0_{\xi}}(F_n)} \int_{\Phi_{\epsilon_{n}}(F_{n})} f^a(u y_0) d(\Phi_{\epsilon_{n}})_{\ast} m_{G^0_{\xi}}(u)\\ &= \frac{1}{m_{G^0_{\xi}}(F_n)}\int_{\Phi_{\epsilon_{n}}(F_{n})} f^a(u y_0) \frac{d(\Phi_{\epsilon_{n}})_{\ast} m_{G^0_{\xi}}}{dm_{G^0_{\xi}}}(u) dm_{G^0_{\xi}}(u)\\
&=\frac{1}{m_{G^0_{\xi}}(\Phi_{\epsilon_{n}}(F_{n}))} \int_{\Phi_{\epsilon_{n}}(F_{n})} f^a(u y_0) dm_{G^0_{\xi}}(u).\\
\end{split}
\end{equation}
Therefore, $(\ref{eq.drift3})$ reads:

\begin{equation}\label{eq.drift6}
\left|\frac{1}{m_{G^0_{\xi}}(\Phi_{\epsilon_{n}}(F_{n}))} \int_{\Phi_{\epsilon_{n}(F_{n})}} f^a(u y_0) dm_{G^0_{\xi}}(u) - \int_X f(x) d\mu(x) \right| <\frac{\delta}{2}.
\end{equation}

Now, by Proposition \ref{prop.variable.change.RN}, $\Phi_{\epsilon_n}(F_n)$ is a subsequence of a tempered F{\o}lner sequence from $\mathcal{C}$. Furthermore, by our choices, $y_0$ is also a generic point for the tempered F{\o}lner sequence $\Phi_{\epsilon_n}(F_n)$. Since $f^a \in C_c(X)$, it follows by Lemma \ref{lemma.unif.gen} that for $n$ sufficiently large, we have

\begin{equation}\label{eq.drift7}
\left| \frac{1}{m_{G^0_{\xi}}(\Phi_{\epsilon_{n}}(F_{n}))} \int_{\Phi_{\epsilon_{n}(F_{n})}} f^a(u y_0) dm_{G^0_{\xi}}(u) - \int_X f^a(x) d\mu(x) \right| <\frac{\delta}{2}.
\end{equation}

From  \eqref{eq.drift6} and \eqref{eq.drift7}, we conclude

\begin{equation}
\left| \int_X f^a(x) d\mu(x)-\int_X f(x) d\mu(x) \right| < \delta.
\end{equation}

Since $\delta>0$ and $f$ are arbitrary, this shows $a$-invariance of $\mu$. Since $\mu$ is $G^0_\xi$-invariant by assumption, its $G_\xi$-invariance follows from $AN$ decomposition $G_\xi=a^\bbZ G_\xi^0$ (see \S \ref{subsub.decompositions}).
\end{proof}

\subsubsection{Mautner phenomenon and $a$-ergodicity of $\mu$}

\begin{lemma}\label{lemma.a.ergodic}
The probability measure $\mu$ is $a$-ergodic.
\end{lemma}

\begin{proof}
Consider the representation of $G_\xi=a^\mathbb{Z}G^0_\xi$ on $L^2(X,\mu)$ given by
$
g: f(x) \mapsto f(g^{-1} x) 
$. By $G_\xi$-invariance of $\mu$, it is a unitary representation.

Let $f \in L^2(X,\mu)$ be an $a$-invariant function. By Mautner phenomenon (Lemma \ref{Mautner}) 
$f$ is invariant under $\overline{U^+}$. Since, by Lemma \ref{lem::closure_contraction_gr}, $\overline{U^+}=G^0_\xi$ and by $G_{\xi}^0$-ergodicity it follows that $f$ is constant, implying $a$-ergodicity.
\end{proof}

 \subsubsection{Constructing $G_{\xi_-}^0$-invariant functions}
 \label{sec::G_xi_inv_fct}

In this section, for each $\theta \in C_c(X)$ we construct a function $\tilde{\theta}: X \to \bbR$ that is $G_{\xi_-}^0$-invariant. This construction goes in two steps. \\[-7pt]

\noindent \textit{Step 1.} We define $\hat{\theta} :X \to \bbR$ by $$  \hat{\theta}(x):=\liminf_{\ell \to \infty} \frac{1}{\ell} \sum_{i=0}^{\ell-1} \theta(a^{\ell} x).$$ 
Note that $\hat{\theta}$ is a bounded, Borel measurable and $a$-invariant function. Let $u\in U^{-}$. Observe that for every $x \in X$, $\varepsilon>0$ and $\ell$ large enough, we have $|\theta(a^{\ell} x)-\theta(a^{\ell}u x)|<\varepsilon$. This follows by writing $\theta(a^{\ell}u x)=\theta(a^{\ell}u a^{-\ell} a^{\ell} x)$, using the facts that $a^{\ell} u a^{-\ell} \to \id $ and uniform continuity of $\theta$. Thus, $\hat{\theta}$ is $U^{-}$-invariant. \\[-7pt]

\noindent \textit{Step 2.} We define $\tilde{\theta} :X \to \bbR$ by $$ \tilde{\theta}(x):=\int_{M} \hat{\theta}(h x)dm_M(h),$$ where $m_M$ denotes the Haar probability measure on the compact group $M:= G_{\xi_-}^0 \cap G_{\xi}^0$.

Again, $\tilde{\theta}$ is bounded and Borel measurable $M$-invariant function. 
Moreover, we claim $\tilde{\theta}$ is $G_{\xi_-}^0$-invariant. By Levi decomposition (\S \ref{subsub.decompositions}), it suffices to show that it is $U^-$-invariant. This follows from $U^-$-invariance of $\hat{\theta}$ and the fact that $M$ normalizes $U^-$.  
Using these, for  $u \in U^-$ we have
\begin{equation*}
\begin{split}
 \tilde{\theta}(u x)=\int_{M} \hat{\theta}(hu x)dm_M(h)=\int_{M} \hat{\theta}(u'h x)dm_M(h)
 =\int_{M} \hat{\theta}(h x)dm_M(h)=\tilde{\theta}(x)  .
 \end{split}
\end{equation*}

\subsubsection{Further invariance properties of $\tilde{\theta}$} 
\label{sec::app_Fubini_full_measures}

\begin{lemma}\label{lem.theta.tilde.mu.constant}
There exists a $G^0_{\xi_-}$-invariant subset $Z \subseteq X$ of full $m_X$-measure on which $\tilde{\theta}$ is constant and equals $c=\int_{X} \theta(x) d\mu(x)$. 
\end{lemma}

\begin{proof}
By $a$-ergodicity of $\mu$ (Lemma \ref{lemma.a.ergodic}) and by Birkhoff's ergodic theorem, $\hat{\theta}$ is $\mu$-a.e. constant and equal to $\int_{X} \theta(x) d\mu(x)=:c$. Let $Z' := \hat{\theta}^{-1}(c)$. Clearly, $\mu(Z')=1$. Since $M \leq G^0_{\xi}$ preserves the measure $\mu$, by Lemma \ref{lemma.fubini} we conclude that for $\mu$-a.e. $x \in X$ and $m_{M}$-a.e. $h \in  M$, we have $h x \in Z'$. By construction of $\tilde{\theta}$, this implies $\tilde{\theta}(x)=c$ for all $x\in Z'$. Denote by $Z:=\tilde{\theta}^{-1}(c)$. Clearly, $\mu(Z)=1$, and $Z$ is $G_{\xi_-}^0$-invariant by $G_{\xi_-}^0$-invariance of $\tilde{\theta}$. 

Now we show that $Z$ has full $m_X$ measure. By Proposition \ref{prop::a_invariance}, $\mu$ is $G_\xi$-invariant. Therefore, applying Lemma \ref{lemma.fubini}, we get that for $\mu$-a.e. $x \in X$, and $m_{G_{\xi}}$-a.e. $b \in G_{\xi}$, we have $b x \in Z$. Fix some $y_0 \in X$ such that  $b  y_0 \in Z$ for $m_{G_{\xi}}$-a.e. $b \in G_{\xi}$. Denote $y_0=:g_0\Gamma$. 

Let $\pi: G \to G/\Gamma = X$ be the quotient map and consider the preimage $\pi^{-1}(Z) = Z\Gamma$. We have $b g_0 \in Z\Gamma$ for  $m_{G_{\xi}}$-a.e. $b \in G_{\xi}$. By $G^0_{\xi_-}$-invariance of $Z$,  $Z\Gamma$ is $G^0_{\xi_-}$-invariant too. It follows that $Z\Gamma $ contains $m_{G^0_{\xi_-}} \otimes m_{G_{\xi}}$-almost all of $G^0_{\xi_-} G_{\xi} g_0$, which is $m_G$-conull in $G$ by Bruhat decomposition. Hence, $Z$ has full $m_X$ measure.
\end{proof}

\subsubsection{$\Gamma$ is a lattice}

\begin{proposition}
\label{prop::lattice}
Let $\Gamma$ be discrete subgroup of $G$ and $X=G/\Gamma$. Let $\mu$ be a $G^0_{\xi}$-invariant and ergodic Borel probability measure on $X$ such that for every $x \in X$ we have $\mu(G_{\xi}^0 x)=0$. Then $\Gamma$ is a lattice of $G$.
\end{proposition}

\begin{proof}
Let $\theta \in C_c(X)$ be a non-negative function that is not identically zero on the support of $\mu$. Let $\tilde{\theta}$ be as constructed in \S \ref{sec::G_xi_inv_fct}. 
By Lemma \ref{lem.theta.tilde.mu.constant},
$\tilde{\theta} \equiv c$ $m_X$-almost everywhere. Observe also that $c>0$.

Using Fubini's Theorem, $G$-invariance of $m_X$ and Fatou's Lemma, we have
\begin{equation*}
\begin{aligned}
c   m_X(X)&=\int_X \tilde{\theta}(y) dm_X(y)= \int_X \left(\int_{M} \hat{\theta}(hy)dm_M(h)\right) dm_X(y)\\
&=\int_{M} \left(\int_X \hat{\theta}(y)dm_X(y)\right) dm_M(h)= \int_X \hat{\theta}(y)dm_X(y)\\
&\leq \liminf_{n \to \infty} \frac{1}{n} \sum_{\ell=0}^{n-1}\int_X \theta(a^\ell y) dm_X(y)\\ 
&=\int_X \theta(y) dm_X(y)<+\infty.\\
\end{aligned}
\end{equation*}
The last inequality is strict since $\theta$ is compactly supported and bounded, and $m_X$ is a Radon measure on $X$. 
Hence $\Gamma$ is a lattice in $G$.
\end{proof}

\subsubsection{End of the proof: $\mu=m_X$}

By Proposition \ref{prop::lattice}, we normalize the finite measure $m_X$ to be a probability measure on $X$.

Since the canonical $G$-action on $(X, m_X)$ is clearly ergodic, by the Howe--Moore property \cites{BM00a,Cio} of $G$, $a$ acts ergodically on $(X,m_X)$.

Let $\theta \in C_c(X)$ and $Z\subseteq X$ be as in Lemma \ref{lem.theta.tilde.mu.constant}, in particular $\tilde{\theta}$ is constant on $Z$ and equals $\int_X \theta(x)d\mu(x)$. By ergodicity of $m_X$ under $a$-action and Birkhoff's Ergodic Theorem, $\hat{\theta}$ is constant  $m_X$-a.e. and equals  $\int_X \theta(x) dm_X(x)$. Since $m_X$ is $M$-invariant, by Lemma \ref{lemma.fubini}, $\tilde{\theta}$ is also constant on a set $Z_1$ with $m_X(Z_1)=1$ and equals $\int_X \theta(x) dm_X(x)$. Since $m_X(Z)=1$, we deduce $Z\cap Z_1\neq \emptyset$ and for $z$ in this intersection, we have
$$
\int_X \theta(x) dm_X(x)= \tilde{\theta}(z) =\int_X \theta(x) d\mu(x).
$$
Since $\theta$ is an arbitrary function in $C_c(X)$, we deduce $\mu=m_X$.

\section{Proof of Proposition \ref{prop.variable.change.RN}}\label{sec.proof.of.prop}

Recall from Bruhat decomposition (\S \ref{subsub.decompositions}) that for an element $g \in G$ with $g \xi \neq \xi_-$, we can write $g$ as a product $va^nu$ where $v \in G^0_{\xi_-}$, $n \in \mathbb{Z}$ and $u \in G^0_\xi$. This factorization, key to the drift argument, is not unique. To be able to make use of it and to calculate the Radon--Nikodym derivatives as in Proposition \ref{prop.variable.change.RN}, we need to fix some appropriate (non-global) sections in this factorization (i.e. $\Psi_{\epsilon}$ and $\Phi_\epsilon$). To do this, we start by describing an algorithm to extend in a unique and coherent manner a given partial automorphism of $T$. This will be done by fixing some total orders on permutation groups and using Tits' independence property.

\subsection{Extensions of partial automorphisms}\label{subsec.canonical_extension}

Let $D\subset T$ be a subtree. Let $\tau:D\to T$ be a graph homomorphism. We say $(D,\tau)$ is a \textbf{partial automorphism} of $G$, if there exists $g\in G$, such that $g\big|_D=\tau$. We also say that such a $g$ extends $(D,\tau)$. In general, there might be many elements $g\in G$ that extend a partial automorphism $(D,\tau)$.

\begin{lemma}\label{lemma.extension}(Coherent choice of extensions) Let $G$ be a closed subgroup of $\Aut(T)$ satisfying Tits' independence property. For every subtree $D \subseteq T$, vertex $o$ of $D$, and partial automorphism $(D,\tau)$ of $G$, one can choose an extension $\tilde{\tau}=\tilde{\tau}(D,\tau,o)\in G$ such that the family $\{\tilde{\tau}(D,\tau,o)\}_{(D,\tau,o)}$ of these extensions is coherent in the sense that for any $C , D \subset T$ containing a common vertex $o \in T$ and any partial automorphisms $(C,\sigma), (D,\tau)$ of $G$, if for some $r \in \mathbb{N}$ we have $C \cap B(o,r)=D \cap B(o,r)$ and $\sigma\big|_{C \cap B(o,r)}=\tau\big|_{D\cap B(o,r)}$, then $\tilde{\sigma}(C,\sigma,o)\big|_{B(o,r)}=\tilde{\tau}(D,\tau,o)\big|_{B(o,r)}$.
\end{lemma}

\begin{proof}


Without loss of generality we assume that the valency of $o$ is $d_0$. For every $n \in \bbN$, denote by $S_n$ the set of vertices in $T$ at distance $n$ from $o \in T$. Fix a total order on $S_n$ for each $n\in \bbN$. At each vertex $v$ of valency $d_i$ color the oriented edges with origin $v$ by $1,2,\dots, d_i$ for $i=0,1$. Fix some total orders on the groups $\Sym(d_i)$. This induces lexicographic orders on $\Sym(d_i)^{S_{2n+i}}$ for $i=0,1$ and  $n\in \bbN$. Further, for every $n \in \mathbb{N}$, set $D_n:=D \cup B(o,n)$. 

By induction, we define a sequence of partial automorphisms $(D_n,\tau_n)$ of $G$ such that $\tau_n\big|_{D_j}=\tau_j$, for any $ 0 \leq j < n$ as follows: for $n=0$, define $\tau_0:D=D_0 \to T$ by $\tau_0=\tau$. Assume we constructed a partial automorphism $(D_n,\tau_n)$ of $G$. Consider the set
\[ \Sigma_n:= \left\{  (\pi_v)_{v \in S_n} \in \Sym(d_{n(mod \; 2)})^{S_n} \; \vert \; \exists g\in G \, \text{s.t.\ } g\big|_{D_n}=\tau_n, \; g_*(v)= \pi_v , \forall v\in S_n \right\} .\]
Since by induction hypothesis, $(D_n,\tau_n)$ is a partial automorphism of $G$, the finite set $\Sigma_n$ is non-empty. 
Let $(\pi_v^{\min})_{v \in S_n}$ be the minimal element of $\Sigma_n$ 
with respect to the order we fixed on $\Sym(d_{n(mod \; 2)})^{S_n}$. Now, let $\tau_{n+1}$ be the extension of $\tau_n$ to $D_{n+1}$ by the minimal element $(\pi_v^{\min})_{v \in S_n}$. This clearly defines a partial automorphism $(D_{n+1},\tau_{n+1})$ of $G$. 

Now, for each $n \in \mathbb{N}$, let $g_n \in G$ be such that $g_n\big|_{D_n}=\tau_n$. Since $G$ is closed and $D_n$'s exhaust $T$, the sequence $g_n$ converges in $G$. Define $\tilde{\tau}:=\lim_{n \to \infty} g_n$.

To prove the claim about the coherence of the family of extensions, let $(C,\sigma), (D,\tau)$ be two partial automorphisms with $o \in C \cap D$ and $r \in \mathbb{N}$ as in the statement. Let $\sigma_n$ and $\tau_n$ be the partial automorphisms defined inductively on $C_n$ and $D_n$ respectively, as in the construction above. It is enough to show  $\sigma_r|_{B(o,r)} = \tau_r|_{B(o,r)}$. We will show by induction that $\sigma_n|_{B(o,n)} = \tau_n|_{B(o,n)}$ for all $0 \leq n\leq r$. For $n=0$, it is obvious. Assume this for $n<r$. To show for $n+1$, we need to show that the set of permutations $\Sigma_n$ defined above is the same for $\tau_n$ and $\sigma_n$. This will follow if we show that for every $g\in G$ extending $\sigma_n$, there exists $h\in G$ extending $\tau_n$, so that $g|_{B(o,n+1)}=h|_{B(o,n+1)}$ (and vice-versa, i.e. when roles of $g$ and $h$ are reversed). Fix such $g$ and let $h'$ be any extension of $\tau_n$. Since $C$ and $D$ agree on $B(o,n+1)$, the element $h'^{-1}g$ stabilizes pointwise $B(o,n) \cup ((C \cup D) \cap B(o,n+1))$.  Tits' independence property guarantees that we can find $s\in G$ that acts as $h'^{-1}g$ on $B(o,n+1)$ and trivially on $D \cup B(o,n)$. Then, $h=h's$ is the desired extension of $\tau_n$.
\end{proof}

\begin{remark}\label{rem::choosing.extension}(Fixing a family of extensions)
Note, while Lemma~\ref{lemma.extension} states existence of coherent extensions, its proof in fact describes a certain algorithm for extending partial automorphisms, once the total orders and the basepoint $o$ are fixed. We will always use the coherent extension from the proof of Lemma~\ref{lemma.extension}, by explicitly specifying the choice of $o$ in each application. \end{remark}

\subsection{The maps $\Psi_\epsilon$}
\label{subsec.def_psi}
Here, using the construction of extensions in \S \ref{subsec.canonical_extension}, we construct partially defined sections for the factors in Bruhat decomposition. 

In the sequel of this section, let $G \leq \Aut(T)$ be a closed subgroup satisfying Tits' independence property and (flip), and that is $2$-transitive on $\partial T$. For $i \in \mathbb{Z}$, let $F_i
:= G_{[x_i,\xi)} \setminus G_{[x_{i-1},\xi)} \subset G_\xi^0$. These are compact subsets of $G$. To fix further notation, for an element $\epsilon \in G$, define
\begin{equation}\label{def_r(g)}
    r(\epsilon):= \sup \{ i \in \mathbb{Z} \; \vert \; \epsilon \; \text{fixes the edge} \; [x_{i-1},x_i] \},  
\end{equation}
with the convention that $\sup \emptyset =-\infty$.

Note that $r(\epsilon)\neq + \infty$ if and only if $\epsilon \notin G^0_\xi$. For the following definition, observe that for an element $\epsilon$ with $r(\epsilon) \neq \pm \infty$ and $u \in F_{r(\epsilon)+1}$, we have $u(x_{r(\epsilon)+1})=x_{r(\epsilon)+1}$ and $u(x_{r(\epsilon)}) \neq x_{r(\epsilon)}$. In particular, $x_{r(\epsilon)+1}$ belongs to the geodesic $(\xi,u\epsilon(\xi))$ and $x_{r(\epsilon)}$ does not (see Fig. 1).

\begin{definition}[The maps $\Psi_\epsilon$]\label{def.psi}
Given an element $\epsilon \in G$ with $r(\epsilon) \neq \pm \infty$ and setting $i:=r(\epsilon)+1$, define the map $$\Psi_\epsilon:F_i \to G$$ as follows:  given $u \in F_i$, let
$$D:=T_{[x_i,x_{i-1}]} \cup [x_i,\xi) \cup [x_i,u\epsilon(\xi)), \; \; \; o:= x_{i}$$
(see Definition~\ref{def::Tits_prop} for the notation  $T_{[x_i,x_{i-1}]}$) 
and consider the graph homomorphism $\tau : D \to T$ that fixes pointwise the half-tree $T_{[x_i,x_{i-1}]}$, interchanges the two rays $[x_i,\xi), [x_i,u\epsilon(\xi))$ and $\tau(x_i)= x_i$. Note that $\tau$ is a graph homomorphism because of the observation in the paragraph before the definition (see Fig. 1 for an illustration). As $G$ satisfies Tits' independence property and (flip), there exists $g \in G_{[x_{i-1},x_i]}$ such that $g\big|_D=\tau$, in other words, $(D,\tau)$ is a partial automorphism of $G$. Therefore, by Remark \ref{rem::choosing.extension}, with $o=x_i$, we get an extension $\tilde{\tau} \in G$ of $(D,\tau)$, and we set $\Psi_{\epsilon}(u):=\tilde{\tau} \in G$.
\end{definition}

The following lemma summarizes some key properties of the map $\Psi_\epsilon$:

\begin{lemma}\label{properties_psi}
The map $\Psi_{\epsilon}:F_i \to G$ satisfies the following properties for every $u \in F_i$:\\[3pt]
\indent $(a)$ $\Psi_\epsilon(u)(x_i)=x_i$,\\[2pt]
\indent $(b)$ $\Psi_\epsilon(u) $ fixes pointwise the half-tree $ T_{[x_i,x_{i-1}]}$,\\[2pt]
\indent$(c)$ $\Psi_\epsilon(u)(\xi)=u\epsilon(\xi)$ and $\Psi_\epsilon(u)(u\epsilon(\xi))=\xi$,\\[2pt] 
\indent$(d)$ For $r \in \mathbb{N}$, if the actions of $u_1, u_2 \in F_i$ agree on $B(x_i,r)$, then so do the actions of $\Psi_\epsilon(u_1)$ and $\Psi_\epsilon(u_2)$ on $B(x_i,r)$. In particular, $\Psi_\epsilon$ is continuous,\\[2pt] 
\indent$(e)$ for $u_1, u_2 \in F_i$, we have $\Psi_\epsilon(u_1)=\Psi_\epsilon(u_2)  \iff u_1\epsilon(\xi)=u_2\epsilon(\xi)$.
\end{lemma}

\begin{proof}
The properties (a)-(d) are satisfied by Definition \ref{def.psi}. For (e), the implication ``$\Rightarrow$'' is just first part of (c). The other implication ``$\Leftarrow$'' follows from the unique choice of the extension $\Psi_{\epsilon}(u):=\tilde{\tau} \in G$.
\end{proof}

\tikzset{every picture/.style={line width=0.75pt}} 

\begin{figure}[t]
\scalebox{.85}{
\begin{tikzpicture}[x=0.75pt,y=0.75pt,yscale=-1,xscale=1]


\draw    (108.5,161) -- (451.5,162) (148.51,157.12) -- (148.49,165.12)(188.51,157.23) -- (188.49,165.23)(228.51,157.35) -- (228.49,165.35)(268.51,157.47) -- (268.49,165.47)(308.51,157.58) -- (308.49,165.58)(348.51,157.7) -- (348.49,165.7)(388.51,157.82) -- (388.49,165.82)(428.51,157.93) -- (428.49,165.93) ;

\draw    (268.5,162) -- (183.5,27) (243.27,129.44) -- (250.04,125.17)(221.42,94.74) -- (228.19,90.48)(199.58,60.04) -- (206.35,55.78) ;

\draw    (484.5,17) -- (315.5,124) -- (309,163) (452.84,41.78) -- (448.56,35.02)(419.05,63.17) -- (414.77,56.42)(385.25,84.57) -- (380.97,77.81)(351.46,105.97) -- (347.18,99.21)(317.66,127.37) -- (313.38,120.61) ;

\draw    (315.5,124) -- (293.5,24) (303,85.79) -- (310.81,84.07)(294.4,46.73) -- (302.22,45.01) ;

\draw    (410.5,94) .. controls (450.1,64.3) and (480.88,135.55) .. (423.27,135.04) ;
\draw [shift={(421.5,135)}, rotate = 361.90999999999997] [color={rgb, 255:red, 0; green, 0; blue, 0 }  ][line width=0.75]    (10.93,-3.29) .. controls (6.95,-1.4) and (3.31,-0.3) .. (0,0) .. controls (3.31,0.3) and (6.95,1.4) .. (10.93,3.29)   ;

\draw    (410.5,94) -- (405.24,97.01) ;
\draw [shift={(403.5,98)}, rotate = 330.26] [color={rgb, 255:red, 0; green, 0; blue, 0 }  ][line width=0.75]    (10.93,-3.29) .. controls (6.95,-1.4) and (3.31,-0.3) .. (0,0) .. controls (3.31,0.3) and (6.95,1.4) .. (10.93,3.29)   ;

\draw (267,180) node [scale=0.9] [align=left] {$\displaystyle x_{r(\epsilon)}$};
\draw (310,180) node [scale=0.9] [align=left] {$\displaystyle x_{i}$};
\draw (229,179) node [scale=0.9] [align=left] {$\displaystyle x_{i-2}$};
\draw (351,178) node [scale=0.9] [align=left] {$\displaystyle x_{i+1}$};
\draw (188,179) node [scale=0.9] [align=left] {$\displaystyle x_{i-3}$};
\draw (393,180) node [scale=0.9] [align=left] {$\displaystyle x_{i+2}$};
\draw (484,160) node  [align=left] {$\displaystyle \xi$};
\draw (216,20) node  [align=left] {$\displaystyle \epsilon ( \xi)$};
\draw (345,129) node [scale=0.8] [align=left] {$\displaystyle u( x_{i-1})$};
\draw (80,163) node  [align=left] {$\displaystyle \xi _{-}$};
\draw (516,33) node  [align=left] {$\displaystyle u\epsilon ( \xi)$};
\draw (326,21) node  [align=left] {$\displaystyle u( \xi _{-})$};
\draw (466,99) node   {$\tau $};
\end{tikzpicture}}
\caption{Definition of the map $\Psi_\epsilon$}
\end{figure}
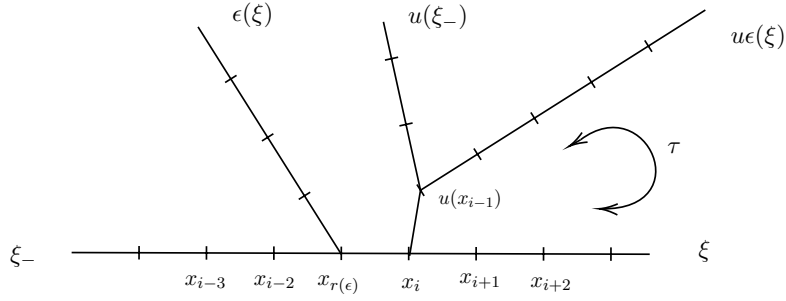

\subsection{The maps $\Phi_\epsilon$ and their properties}
\label{sec::map_Phi}
Here, using the maps $\Psi_\epsilon$ from \S \ref{subsec.def_psi}, we define the maps $\Phi_\epsilon$ that we use for change of variable in the drift argument. In particular, it is crucial to know that $\Phi_\epsilon$ is injective. Further, we identify the image of $\Phi_\epsilon$; this will be important for the calculation of the Radon--Nikodym derivative ((4) of Proposition \ref{prop.variable.change.RN}) in the following subsection (\S \ref{subsec.RN}).

In the sequel, let us fix an element $\epsilon \in G$ with $r(\epsilon) \neq \pm \infty$ as in Definition \ref{def.psi} and denote $i:= r(\epsilon)+1$. We define the map 
  \begin{equation}\label{eq.def_of_Phi}
     \Phi_\epsilon: F_i \to G,  \; \; \;  u \mapsto \Phi_\epsilon(u):= a^{-1}\Psi_\epsilon(u)^{-1}u\epsilon.
    \end{equation}

Being a composition of continuous maps,  $\Phi_\epsilon$ is continuous. In what follows, we identify the image of $\Phi_\epsilon$ and show that this map is injective. 

\subsubsection{Pinpointing the image of $\Phi_\epsilon$} To proceed, we fix an element $t_\epsilon \in G_{[x_{i-1},\xi)}$ such that 
 \begin{equation}\label{eq.def_of_t}
     t_\epsilon^{-1}([x_{i-3}, x_{i-1}])= \epsilon^{-1}([x_{i-1}, x_{i+1}]).
 \end{equation}
 
Such element $t_\epsilon$ exists due to the $2$-transitivity of $G$ on the boundary $\partial T$ (for an illustration see Fig. 2).

\tikzset{every picture/.style={line width=0.75pt}} 
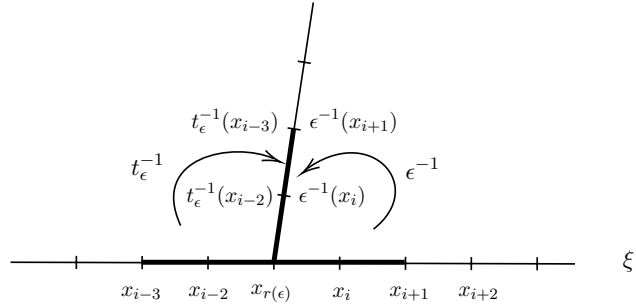
\begin{figure}[t]
\scalebox{.83}{
\begin{tikzpicture}[x=0.75pt,y=0.75pt,yscale=-1,xscale=1]

\draw    (108.5,161) -- (451.5,162) (148.51,157.12) -- (148.49,165.12)(188.51,157.23) -- (188.49,165.23)(228.51,157.35) -- (228.49,165.35)(268.51,157.47) -- (268.49,165.47)(308.51,157.58) -- (308.49,165.58)(348.51,157.7) -- (348.49,165.7)(388.51,157.82) -- (388.49,165.82)(428.51,157.93) -- (428.49,165.93) ;

\draw    (268.5,161) -- (292.5,3)(270.7,119.86) -- (278.61,121.07)(276.86,79.33) -- (284.77,80.53)(283.02,38.79) -- (290.93,40) ;

\draw    (268.5,98) .. controls (243.75,87.11) and (192.54,96.8) .. (211.89,138.72) ;
\draw [shift={(272.5,100)}, rotate = 200.97] [color={rgb, 255:red, 0; green, 0; blue, 0 }  ][line width=0.75]    (10.93,-3.29) .. controls (6.95,-1.4) and (3.31,-0.3) .. (0,0) .. controls (3.31,0.3) and (6.95,1.4) .. (10.93,3.29)   ;

\draw    (287.5,105) .. controls (326.9,75.45) and (362.42,114.79) .. (329.07,140.82) ;
\draw [shift={(285.5,107)}, rotate = 324.15999999999997] [color={rgb, 255:red, 0; green, 0; blue, 0 }  ][line width=0.75]    (10.93,-3.29) .. controls (6.95,-1.4) and (3.31,-0.3) .. (0,0) .. controls (3.31,0.3) and (6.95,1.4) .. (10.93,3.29)   ;

\draw [line width=2.25]    (280.5,80) -- (268.5,161) ;

\draw [line width=2.25]    (188,161) -- (268.5,161) ;

\draw [line width=2.25]    (268,161) -- (348,161) ;

\draw (267,180) node [scale=0.9] [align=left] {$\displaystyle x_{r(\epsilon)}$};
\draw (310,180) node [scale=0.9] [align=left] {$\displaystyle x_{i}$};
\draw (229,180) node [scale=0.9] [align=left] {$\displaystyle x_{i-2}$};
\draw (351,180) node [scale=0.9] [align=left] {$\displaystyle x_{i+1}$};
\draw (188,180) node [scale=0.9] [align=left] {$\displaystyle x_{i-3}$};
\draw (393,180) node [scale=0.9] [align=left] {$\displaystyle x_{i+2}$};
\draw (359,104) node  [align=left] {$\displaystyle \epsilon^{-1}$};
\draw (192,103) node  [align=left] {$\displaystyle t_{\epsilon }^{-1}$};
\draw (484,160) node  [align=left] {$\displaystyle \xi$};
\draw (304,120) node [scale=0.9] [align=left] {$\displaystyle \epsilon^{-1}(x_i)$};
\draw (242,120) node [scale=0.9] [align=left] {$\displaystyle t_{\epsilon}^{-1}(x_{i-2})$};
\draw (317,75) node [scale=0.9] [align=left] {$\displaystyle \epsilon^{-1}(x_{i+1})$};
\draw (245,75) node [scale=0.9] [align=left] {$\displaystyle t_{\epsilon}^{-1}(x_{i-3})$};

\end{tikzpicture}}
\caption{The element $t_\epsilon \in G_{[x_{i-1},\xi)}$}
\end{figure}

For later use, we note the following straightforward observation:

\begin{lemma} \label{lem::tig_no_matter} For any $h \in G_{[x_{i-1}, \xi)}$ satisfying $h^{-1}([x_{i-3}, x_{i-1}])= \epsilon^{-1}([x_{i-1}, x_{i+1}])$, we have $F_{i-2}  h= F_{i-2}  t_\epsilon$.
\end{lemma}

\begin{proof}
Indeed, it follows from the hypothesis that $t_\epsilon h^{-1}  \in G_{[x_{i-3},\xi)}$. Therefore, since $F_{i-2}$ is a union of left cosets of $G_{[x_{i-3},\xi)}$ (see Example \ref{ex.F_i}), we have $F_{i-2}t_\epsilon h^{-1} \subset F_{i-2}$. By symmetry, we are done.
\end{proof}


The following lemma is the first step to identify the image of the map $\Phi_\epsilon$. 

\begin{lemma}
\label{lemma:properties_Phi}
 Let $t_\epsilon \in G_{[x_{i-1},\xi)}$ be such that \eqref{eq.def_of_t} is satisfied. Then 
 \[ \Phi_\epsilon (F_i) \subseteq F_{i-2}  t_\epsilon.\]
\end{lemma}

\begin{proof}
By definitions of the map $\Phi_\epsilon$ and the set $F_{i-2}$, one only needs to check that the product $a^{-1}\Psi_\epsilon(u)^{-1} u \epsilon t^{-1}_\epsilon$ fixes pointwise the geodesic ray $[x_{i-2}, \xi)$, but not the vertex $x_{i-3}$. This is a routine verification that follows from the choice of $t_\epsilon$ and the properties of $\Psi_\epsilon$ in Lemma~\ref{properties_psi}, see Fig. 1 and 2. 
\end{proof}

As mentioned earlier, the following result is crucial for the change of variables in the drift argument. We note that neither injectivity, nor the surjectivity of $\Phi_\epsilon$ is entirely clear from its defining formula~\eqref{eq.def_of_Phi}, especially as by (e) of Lemma \ref{properties_psi}, the map $\Psi_\epsilon$ fails to be injective.

\begin{prop}\label{phi_is_bijection}
The map $\Phi_\epsilon : F_i \to F_{i-2}  t_\epsilon$ is a  continuous bijection from the compact set $F_i$ onto $F_{i-2}  t_\epsilon$, in particular, it is a homeomorphism.
\end{prop}
\begin{proof}
Continuity of $\Phi_\epsilon$ is noted above; we only need to show its bijectivity. Let us first show the injectivity. To do this, assume that for $u_1, u_2 \in F_i$, we have $\Phi_\epsilon(u_1)=\Phi_\epsilon(u_2)$. Unfolding the definition \eqref{eq.def_of_Phi} of $\Phi_\epsilon$, this is equivalent to
\begin{equation}\label{eq.phi.bij1}
\Psi_\epsilon(u_1)\Psi_\epsilon(u_2)^{-1}=u_1 u_2^{-1}.
\end{equation}
Since $u_1u_2^{-1} \in G^0_\xi$, this gives $\Psi_\epsilon(u_1)^{-1}(\xi)=\Psi_\epsilon(u_2)^{-1}(\xi)$, and therefore by (c) from Lemma~\ref{properties_psi}, we get $u_1\epsilon(\xi) =u_2\epsilon(\xi)$. Finally, in view of (e) of Lemma~\ref{properties_psi}, this shows  $\Psi_\epsilon(u_1)=\Psi_\epsilon(u_2)$ and thus by $(\ref{eq.phi.bij1})$, $u_1=u_2$, proving injectivity.

Next, we show surjectivity. To do this, given $u' \in F_{i-2}t_\epsilon$, we construct an element $v$ in the range of $\Psi_\epsilon$ such  that $vau'\epsilon^{-1} \in F_i$. Then, setting $u:=vau'\epsilon^{-1}$, we have by $(e)$ Lemma \ref{properties_psi}, $\Psi_\epsilon(u)=v$ and by definition $(\ref{eq.def_of_Phi})$ of $\Phi_\epsilon$, $\Phi_\epsilon(u)=u'$, proving surjectivity. To construct such a $v$, notice that the geodesic between $\xi$ and $au'\epsilon^{-1}(\xi)$ passes through $x_i$, but not $x_{i-1}$ (see Fig. 3). Then, as in Definition \ref{def.psi} of $\Psi_\epsilon$, setting
$$D:= T_{[x_i,x_{i-1}]} \cup [x_i,\xi) \cup [x_i, au'\epsilon^{-1}(\xi)), \; \; \; o:= x_i,$$ and considering the graph homomorphism $\tau: D \to T$ fixing pointwise the half tree $T_{[x_i, x_{i-1}]}$, interchanging the two rays $[x_i,\xi)$, $[x_i, au'\epsilon^{-1}(\xi))$, and $\tau(x_i)= x_i$, we get a partial automorphism $(D,\tau)$ of $G$. By construction of $\Psi_\epsilon$, $(D,\tau)$ has a unique extension $v$ in the range of $\Psi_\epsilon$.

It remains to verify that $u:=vau'\epsilon^{-1} \in F_i$, i.e. $u$ fixes $[x_i,\xi)$ but not $x_{i-1}$. This is routinely verified and we are done.
\end{proof}

\tikzset{every picture/.style={line width=0.75pt}} 
\begin{figure}[t]
\scalebox{.75}{
\begin{tikzpicture}[x=0.75pt,y=0.75pt,yscale=-1,xscale=1]

\draw    (108.5,161) -- (451.5,162) (148.51,157.12) -- (148.49,165.12)(188.51,157.23) -- (188.49,165.23)(228.51,157.35) -- (228.49,165.35)(268.51,157.47) -- (268.49,165.47)(308.51,157.58) -- (308.49,165.58)(348.51,157.7) -- (348.49,165.7)(388.51,157.82) -- (388.49,165.82)(428.51,157.93) -- (428.49,165.93) ;

\draw    (268.5,162) -- (183.5,27) (243.27,129.44) -- (250.04,125.17)(221.42,94.74) -- (228.19,90.48)(199.58,60.04) -- (206.35,55.78) ;

\draw    (396.5,203) .. controls (423.37,238.82) and (368.06,297.41) .. (338.94,235.94) ;
\draw [shift={(338.5,235)}, rotate = 425.28] [color={rgb, 255:red, 0; green, 0; blue, 0 }  ][line width=0.75]    (10.93,-3.29) .. controls (6.95,-1.4) and (3.31,-0.3) .. (0,0) .. controls (3.31,0.3) and (6.95,1.4) .. (10.93,3.29)   ;

\draw    (396.5,203) -- (391.39,192.79) ;
\draw [shift={(390.5,191)}, rotate = 423.43] [color={rgb, 255:red, 0; green, 0; blue, 0 }  ][line width=0.75]    (10.93,-3.29) .. controls (6.95,-1.4) and (3.31,-0.3) .. (0,0) .. controls (3.31,0.3) and (6.95,1.4) .. (10.93,3.29)   ;

\draw    (188.5,161) -- (103.5,26) (163.27,128.44) -- (170.04,124.17)(141.42,93.74) -- (148.19,89.48)(119.58,59.04) -- (126.35,54.78) ;

\draw    (228,161) -- (254.5,192) -- (170.5,279) (257.03,188.81) -- (250.95,194)(230.14,222.99) -- (224.38,217.43)(202.35,251.77) -- (196.6,246.21)(174.57,280.54) -- (168.81,274.99) ;

\draw    (309,162) -- (335.5,193) -- (251.5,280) (338.03,189.81) -- (331.95,195)(311.14,223.99) -- (305.38,218.43)(283.35,252.77) -- (277.6,247.21)(255.57,281.54) -- (249.81,275.99) ;

\draw (267,180) node [scale=0.9] [align=left] {$\displaystyle x_{r(\epsilon)}$};
\draw (310,180) node [scale=0.9] [align=left] {$\displaystyle x_{i}$};
\draw (229,179) node [scale=0.9] [align=left] {$\displaystyle x_{i-2}$};
\draw (351,178) node [scale=0.9] [align=left] {$\displaystyle x_{i+1}$};
\draw (188,179) node [scale=0.9] [align=left] {$\displaystyle x_{i-3}$};
\draw (393,180) node [scale=0.9] [align=left] {$\displaystyle x_{i+2}$};
\draw (484,160) node  [align=left] {$\displaystyle \xi$};
\draw (216,20) node  [align=left] {$\displaystyle \epsilon ^{-1}( \xi)$};
\draw (80,163) node  [align=left] {$\displaystyle \xi _{-}$};
\draw (425,240) node   {$\tau $};
\draw (71,21) node  [align=left] {$\displaystyle t_{\epsilon } \epsilon ^{-1}( \xi)$};
\draw (121,274) node  [align=left] {$\displaystyle u' \epsilon ^{-1}( \xi)$};
\draw (330,272) node  [align=left] {$\displaystyle au'\epsilon ^{-1}( \xi)$};

\end{tikzpicture}}
\caption{Surjectivity of $\Phi_\epsilon$}
\end{figure}
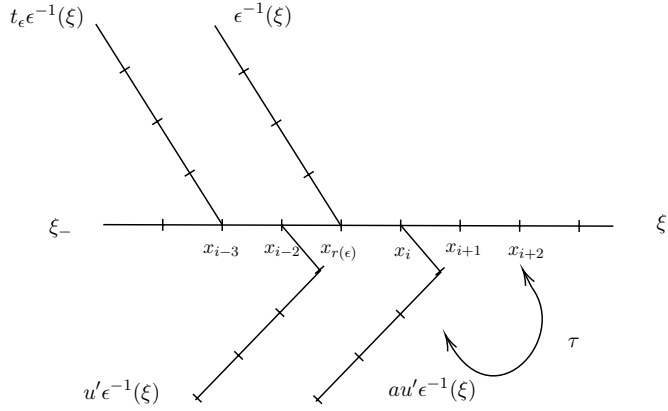

\subsection{The Radon--Nikodym derivative induced by $\Phi_\epsilon$}\label{subsec.RN}

After identifying the image of $\Phi_\epsilon$ our goal is now to understand its local behavior. In the following lemma, we identify the images of ``small" subsets of the domain, which will eventually allow us to compute the Radon--Nikodym derivative using Proposition \ref{prop.differentiation_of_measures}.

\begin{lemma}\label{images_under_phi}
For every $u \in F_i$ and positive integer $r$, we have
\[  \Phi_\epsilon(u G_{B(x_i,r), \xi}) =  G_{B(x_{i-2},r), \xi}  \Phi_\epsilon(u). \]
\end{lemma}

\begin{proof}
The statement of the lemma is equivalent to showing that the map $\tilde{f}$ defined on $G_{B(x_{i},r), \xi}$, by $\tilde{f}(k):=\Phi_\epsilon(uk) \Phi_\epsilon^{-1}(u)$ is injective onto $G_{B(x_{i-2},r), \xi}$. The injectivity is immediate from Proposition \ref{phi_is_bijection}. Thus, we only need to identify its image with $G_{B(x_{i-2},r), \xi}$. Equivalently, we shall show that the image of the function $a\tilde{f}a^{-1}=:f$ is $G_{B(x_{i},r), \xi}$. By its definition, it is plain that $f$ takes values in $G^0_\xi$. Furthermore, unfolding the definitions, we have 
\begin{equation}\label{eq.rn11}
f(k)=\Psi_\epsilon^{-1}(uk)uku^{-1}\Psi_\epsilon(u).
\end{equation}
Since $k \in G_{B(x_{i},r), \xi}$, the actions of $uk$ and $u$ agree on $B(x_{i},r)$, hence, by $(d)$ of Lemma \ref{properties_psi} so do the actions of $\Psi_\epsilon(uk)$ and $\Psi_\epsilon(u)$. In view of $(\ref{eq.rn11})$, it follows that $f(k) \in G_{B(x_{i},r), \xi}$, i.e. $Im(f) \subseteq G_{B(x_{i},r), \xi}$.

It remains to show that the image of $f$ is equal to $G_{B(x_{i},r), \xi}$. To do this, let $k_0 \in G_{B(x_{i},r), \xi}$, we need to find an element $k \in G_{B(x_{i},r), \xi}$ with $f(k)=k_0$, i.e. $\Psi_\epsilon^{-1}(uk)uku^{-1}\Psi_\epsilon(u)=k_0$. Since $u \in G^0_\xi$ and $G^0_\xi$ normalizes $G_{B(x_{i},r), \xi}$, it is equivalent to find $k \in G_{B(x_{i},r), \xi}$ with $\Psi_\epsilon^{-1}(ku)k\Psi_\epsilon(u)=k_0$, equivalently, $k=\Psi_\epsilon(ku) k_0 \Psi_\epsilon(u)^{-1}$. This equation is seen to possess a solution by Tits' independence property: indeed, by this property, we can define an element $k \in G$ satisfying $ku\epsilon \xi =k_0 u \epsilon \xi$ and elsewhere acting as $\Psi_\epsilon(k_0u) k_0 \Psi_\epsilon(u)^{-1}$. Then, by (e) of Lemma \ref{properties_psi}, we have $\Psi_\epsilon(ku)=\Psi_\epsilon(k_0u)$ and $(d)$ of the same lemma, we have $k \in G_{B(x_{i},r), \xi}$ and we are done.
\end{proof}

Let $m_{G^0_{\xi}}\vert_{F_{i-2} t_{\epsilon}}$ be the restriction of $m_{G^0_{\xi}}$ to the compact set $F_{i-2} t_{\epsilon}=\Phi_\epsilon(F_i)$. Here, we compute the Radon--Nikodym derivative of the measure $(\Phi_\epsilon)_*( m_{G_\xi^0}\vert_{F_i})$ with respect to $m_{G^0_{\xi}}\vert_{F_{i-2} t_{\epsilon}}$.

\begin{cor}\label{Radon_Nikodym}
For $m_{G^0_\xi}$-almost every $x\in F_{i-2}   t_{\epsilon}$, we have
\begin{equation}\label{eq.rn1}
 \frac{d(\Phi_\epsilon)_*m_{G_\xi^0}}{dm_{G_\xi^0}}(x) =\frac{m_{G^0_{\xi}}(F_{n})}{m_{G^0_{\xi}}(\Phi_{\epsilon_n}(F_{n}))}.
 \end{equation}
\end{cor}

\begin{proof}
By Lemma~\ref{images_under_phi} and unimodularity of $G_\xi^0$, for every $u\in F_i$ and positive integer $r$, we have
\[ \frac {m_{G_\xi^0}(\Phi_\epsilon^{-1}( G_{B(x_{i-2},r), \xi}\Phi_\epsilon(u)))}{m_{ G_\xi^0}( G_{B(x_{i-2},r), \xi}\Phi_\epsilon(u))} = \frac{m_{ G_\xi^0}(u G_{B(x_{i},r), \xi})}{m_{ G_\xi^0}( G_{B(x_{i-2},r), \xi}\Phi_\epsilon(u))}=(d_0-1)(d_1-1). \]

Moreover, by Lemma \ref{images_under_phi}, it is easy to see  the pushforward measure $(\Phi_\epsilon)_* m_{G_\xi^0}$  $\Phi_\epsilon$ on $F_{i-2}   t_{\epsilon}$ is absolutely continuous with respect to the measure $m_{G^0_{\xi}}\vert_{F_{i-2}   t_{\epsilon}}$. Therefore, since the right-hand side of~\eqref{eq.rn1} equals to $(d_0-1)(d_1-1)$, the claim follows by Proposition \ref{prop.differentiation_of_measures}.\end{proof}
 

\subsection{Proof of Proposition \ref{prop.variable.change.RN}}
\label{subsec::proof_change.RN}

First, note that since $\epsilon_n$ converges to identity, we have $r(\epsilon_n)\in \bbN$ for large enough $n$ (for the notation, see \S \ref{subsec.def_psi}). Up to reindexing and extracting a subsequence one can arrange $r(\epsilon_n)=n-1$. 

Recall that $F_n$ is the sequence in $\mathcal{C}$ given by $F_n=G_{[x_n,\xi)} \setminus G_{[x_{n-1},\xi)}$, where $\mathcal{C}$ is the finite collection of tempered Folner sequences, as defined in \eqref{def.folner.collection}. The maps $\Psi_{\epsilon_n}: F_{n} \to G$ and $\Phi_{\epsilon_n}: F_{n} \to  G$ are then given by Definition~\ref{def.psi} and~\eqref{eq.def_of_Phi}, respectively. Continuity of $\Psi_{\epsilon_n}$ and $\Phi_{\epsilon_n}$ as well as  injectivity of $\Phi_{\epsilon_n}$ follows from Lemma \ref{properties_psi} (d) and Proposition \ref{phi_is_bijection} respectively. \eqref{prop.variable.change.RN_1} follows by definition \eqref{eq.def_of_Phi}.  \eqref{prop.variable.change.RN_2} follows from Lemma~\ref{properties_psi} (b). \eqref{prop.variable.change.RN_4} is Corollary \ref{Radon_Nikodym}.

Finally, to see \eqref{prop.variable.change.RN_3}, recall by Lemma \ref{phi_is_bijection}  there exists $t_{\epsilon_n}\in G_{[x_{n-1},\xi)}$ such that $\Phi_{\epsilon_n}(F_n) = F_{n-2}t_{\epsilon_n}$. Without loss of generality, assume that $\epsilon_n$ contains only even indices. Up to passing to subsequence, for all $n$, the actions of the elements $a^{-n/2}t_{\epsilon_{n}} a^{n/2}$ are the same on $B(x_{-1},2)$, and let $t_0 \in C_{even}$ be an element acting in the same way on $B(x_{-1},2)$. Then by Lemma \ref{lem::tig_no_matter}, we have $\Phi_{\epsilon_{n}}(F_{n}) =F_{n-2}t_{\epsilon_{n}}= a^{n/2} F_{-2} t_0 a^{-n/2}$, and this is in fact a subsequence of an element in $\mathcal{C}$. Similar argument works if $\epsilon_n$ contains a subsequence with odd indices.

\section{Horospherical orbits in geometrically finite quotients}\label{sec.geo.fin}

The aim of this section is to prove Theorem \ref{thm.geo.fin}. In \S \ref{subsec.geo.fin}, we discuss geometrically finite lattices. In \S \ref{subsec.compact.orbits}, we characterize compact orbits and, finally, in \S \ref{subsec.density} we show that the remaining orbits are dense.

\subsection{Geometrically finite lattices in biregular trees}\label{subsec.geo.fin}

\subsubsection{Geometrically finite lattices}

A graph of groups $(\mathcal{G},\mathcal{A})$ is the data of\\
\indent - a graph $\mathcal{G}$,\\  
\indent- for every vertex $v$ of $\mathcal{G}$ a group $\mathcal{A}_v$,\\
\indent- for every edge $[v,w]$ of $\mathcal{G}$ a group $\mathcal{A}_{[v,w]}$ such that $\mathcal{A}_{[v,w]}=\mathcal{A}_{[w,v]}$,\\
\indent- for every edge $[v,w]$ an injective morphism $\rho_{[v,w]}: \mathcal{A}_{[v,w]} \to \mathcal{A}_w$.\\

For a subgroup $\Gamma$ of $G$, $\Gamma \backslash T$ denotes the quotient graph. For an edge $[v,w]$ of $\Gamma \backslash T$, taking $\mathcal{A}_v:=\Gamma_{\tilde{v}}$ and $\mathcal{A}_{[v,w]}:=\Gamma_{[\tilde{v},\tilde{w}]}$ for lifts $\tilde{v}, \tilde{w} \in T$ of the vertices $v,w$  gives the graph $\Gamma \backslash T$ a structure of graph of groups. The morphisms $\rho_{[v,w]}$ can be taken as inclusion maps (recall that $G$ acts without edge-inversion). Up to isomorphism, the structure of graph of groups is independent of choice of lifts (see \cite{bass-serre, BL}). We denote by $\Gamma \backslash \backslash T$ the corresponding graph of groups.


Below we describe an important example of graph of groups, whose underlying graph is a ray with vertices labeled by $c(i), i\in \bbN$ .

\tikzset{every picture/.style={line width=0.75pt}} 

\begin{center}
\begin{tikzpicture}[x=0.75pt,y=0.75pt,yscale=-1,xscale=1]

\draw    (140.5,120) -- (404.5,121) (210.51,116.27) -- (210.48,124.27)(280.51,116.53) -- (280.48,124.53)(350.51,116.8) -- (350.48,124.8) ;

\draw    (140.5,120) ;

\draw    (140.5,124) -- (140.5,116) ;

\draw (140,138) node [scale=0.7] [align=left] {$ $$\displaystyle c(0)$};
\draw (213,138) node [scale=0.7] [align=left] {$ $$\displaystyle c(1)$};
\draw (283,138) node [scale=0.7] [align=left] {$ $$\displaystyle c(2)$};
\draw (349,138) node [scale=0.7] [align=left] {$ $$\displaystyle c(3)$};
\draw (441,119) node  [align=left] {$\displaystyle \cdots $};
\draw (141,100) node  [align=left] {$\displaystyle \mathcal{A}_{0}^+$};
\draw (209,100) node  [align=left] {$\displaystyle \mathcal{A}_{1}$};
\draw (282,100) node  [align=left] {$\displaystyle \mathcal{A}_{2}$};
\draw (350,100) node  [align=left] {$\displaystyle \mathcal{A}_{3}$};
\draw (249,138) node  [align=left] {$\displaystyle \mathcal{A}_{1}$};
\draw (318,138) node  [align=left] {$\displaystyle \mathcal{A}_{2}$};
\draw (178,138) node  [align=left] {$\displaystyle \mathcal{A}_{0}$};

\end{tikzpicture}
\end{center}

Assume that all groups in the graph above are finite and satisfy 
\[   \mathcal{A}_0^+  > \mathcal{A}_0 \leq \mathcal{A}_1 \leq \mathcal{A}_2 \leq \mathcal{A}_3 \leq \dots ,\]
\[  [\mathcal{A}_{0}^{+} : \mathcal{A}_{0}]= d_0, \quad [\mathcal{A}_i : \mathcal{A}_{i-1}] = d_{i(mod \; 2)}-1, \; \text{for $i\geq 1$}.  \]

Such a graph of groups is called a \textbf{Nagao ray} \cites{nagao, bass-serre, BL}.

A lattice $\Gamma\leq G$ is said to be \textbf{geometrically finite} (\cite{paulin.geo.fin}) if $\Gamma \backslash \backslash T$ is a union of a finite graph of groups and a finite number of Nagao rays. 

One can show (see e.g. \cite[proof of Thm. 3.4, (1) $\Rightarrow$ (2)]{paulin.geo.fin}) that if $c:\mathbb{N} \to \Gamma \backslash T$ is a Nagao ray in $\Gamma \backslash \backslash T$ and $\tilde{c}$ is a geodesic lift in $T$ of $c$ having $\xi$ as endpoint, the stabilizer $\Gamma_\xi$ acts transitively on the horosphere based at $\xi$ and passing through $\tilde{c}(0)$. 

\subsubsection{Dynamics of ends of $T$}
Recall that the limit set of a discrete subgroup $\Gamma$ of $G$ is the unique minimal subset of $\partial T$ for the $\Gamma$-action. By Proposition \ref{prop.mostow}, the action of a lattice $\Gamma \leq G$ on $\partial T$ is minimal, in particular, its limit set is $\partial T$. 

Given a lattice $\Gamma \leq G$, an end $\xi \in \partial T$ is called \textbf{$\Gamma$-conical} if there exists a sequence of elements $\gamma_n \in \Gamma$ such that for every vertex $v \in T$ and every geodesic ray $c$ in $T$ with endpoint $\xi$, there exists a constant $C>0$ such that $d(\gamma_n v,c) \leq C$ for every $n \in \mathbb{N}$, where $d(\cdot,\cdot)$ denotes the graph distance on $T$. An end $\xi \in \partial T$ is called \textbf{$\Gamma$-bounded parabolic} if the stabilizer $\Gamma_\xi$ acts properly discontinuously and cocompactly on $\partial T \setminus \{\xi\}$.

By Paulin \cite[Theorem 3.4]{paulin.geo.fin}, for a geometrically finite lattice $\Gamma \leq G$, an end $\xi \in \partial T$ in the limit set of $\Gamma$ (which coincides with $\partial T$) is either $\Gamma$-conical or $\Gamma$-bounded parabolic. 
Denote by $p:T \to \Gamma \backslash T$ the canonical projection map. By \cite[Section 3]{paulin.geo.fin}, an end $\xi$ is $\Gamma$-bounded parabolic if and only if any geodesic ray $c$ in $T$ with endpoint $\xi$ contains a subray $c_0$ such that the restriction of $p$ to $c_0$ is an isomorphism of graphs onto its image in $\Gamma \backslash T$ and the image, endowed with its structure of graph of groups induced from $\Gamma \backslash \backslash T$, is a Nagao ray. Furthermore, by \cite[Corollary  3.5]{paulin.geo.fin} the set of distinct $\Gamma$-orbits of $\Gamma$-bounded parabolic points is in natural bijection with the set of ends of $\Gamma \backslash T$.

Finally, for a fixed vertex $v \in \Gamma \backslash T$ and a lift $\tilde{v}$ of $v$ in $T$, we define the map $p_v: G/\Gamma \to \Gamma \backslash T$ by setting $p_v(g\Gamma):=p(g^{-1}\tilde{v})$. It is easily seen that this map is well-defined, independent of the choice of $\tilde{v}$ and continuous with compact fibres.

\subsection{Characterization of compact orbits}\label{subsec.compact.orbits}

Here, we prove the part of Theorem \ref{thm.geo.fin} concerning the characterization of compact $G^0_\xi$-orbits in $X$. If $\Gamma$ is a uniform lattice in $G$,
then all orbits are dense by Corollary \ref{corol.intro.minimal}, implying part 1. of Theorem \ref{thm.geo.fin}. Part 2. holds trivially, since no such orbit is closed. 
Hence for the rest of the proof, we suppose that $\Gamma$ is a non-uniform lattice in $G$ (e.g. the quotient $\Gamma \backslash T$ contains at least one infinite ray).

\begin{proposition}\label{prop.compact.orbits} Let $x \in X$. The $G^0_\xi$-orbit of $x$ is compact if and only if $x=g\Gamma$ is such that $g^{-1}\xi$ is a $\Gamma$-bounded parabolic end.
\end{proposition}
\begin{proof}
We first show if $g^{-1}\xi$ is a $\Gamma$-bounded parabolic end then $G^0_\xi g\Gamma$ is compact. Since $g^{-1}G^0_\xi g=G^0_{g^-1\xi}$, that is equivalent to showing that $g G^0_{g^{-1}\xi}\Gamma$, hence $G^0_{g^{-1}\xi}\Gamma$ is compact. Therefore, it suffices to fix $\xi \in \partial T$ a $\Gamma$-bounded parabolic end and show $G^0_\xi\Gamma$ is compact in $X$ or, equivalently, that $G^0_\xi \cap \Gamma$ is  cocompact in $G^0_\xi$. To do this, it suffices to find a compact set $K \subset G$ such that $G^0_\xi \subset (G^0_\xi \cap \Gamma)K$. Indeed, by \cite[Proposition 3.1.(ii)]{paulin.geo.fin}, there exists a geodesic ray $\tilde{c}$ in $T$ with endpoint $\xi$, whose projection by $p$ injects onto a geodesic ray $c$ in $\Gamma \backslash T$. Since $\Gamma$ is geometrically finite, up to shortening the ray $\tilde{c}$, we can suppose that $c$, endowed with its induced structure of graph of groups, is a Nagao ray in $\Gamma \backslash \backslash T$. For every $g \in G^0_\xi$, $g.\tilde{c}(0)$ belongs to the horosphere based at $\xi$ and passing through $\tilde{c}(0)$. Since $c$ is a Nagao ray, the group $\Gamma \cap G^0_\xi$ acts transitively on this horosphere and hence there exists $\gamma \in \Gamma \cap G^0_\xi$ such that $\gamma g. \tilde{c}(0) =\tilde{c}(0)$. This shows that $G^0_\xi \subseteq (\Gamma \cap G^0_\xi)G_{\tilde{c}(0)}$, as desired.

For the other implication, similarly, it suffices to assume $\xi \in \partial T$ is not $\Gamma$-bounded parabolic and show the orbit $G^0_\xi\Gamma$ is not compact in $X$. To do this, we show  $G^0_\xi \cap \Gamma$ is finite. Indeed, since $\xi$ is not $\Gamma$-bounded parabolic, by Paulin \cite[Corollaire 3.5]{paulin.geo.fin}, for every geodesic ray $\tilde{c}$ in $T$ with endpoint  $\xi$, the value $c(n)$ of its projection $p \circ \tilde{c}=:c$ on $\Gamma \backslash T$ equals a vertex $v$ for infinitely many $n \in \mathbb{N}$.  Let $\tilde{c}$ be such a ray, $c$ its projection, $v$ a vertex such that, up to shortening $\tilde{c}$, $c(0)=v$ and $c(n_k)=v$ with $n_k \to +\infty$, $\tilde{v}=\tilde{c}(0)$. In particular, for each $n_k \in \mathbb{N}$, there is an element $\gamma_k \in \Gamma$ such that $\gamma_k \tilde{v}=\tilde{c}(n_k)$. 

Since $\Gamma$ is discrete, $G_{\tilde{v}} \cap \Gamma$ is finite, let $N \in \mathbb{N}$ be its cardinality. Suppose for a contradiction that $G^0_\xi \cap \Gamma$ is infinite. Since for every $\gamma \in G^0_\xi \cap \Gamma$, $\gamma$ fixes a geodesic subray $[\tilde{c}(n_\gamma),\xi)$ of $\tilde{c}$, there exists $m_0 \in \mathbb{N}$ such that $|G^0_{[\tilde{c}(m_0),\xi)} \cap \Gamma| \geq N+1$. Now fix $k_0 \in \mathbb{N}$ with $n_{k_0} \geq m_0$. Then, by construction, we have $\gamma_{k_0}^{-1} (G^0_{[\tilde{c}(m_0),\xi)} \cap \Gamma) \gamma_{k_0} \subseteq G_{\tilde{v}}\cap \Gamma$, a contradiction to the cardinality assumption. Therefore $G^0_\xi \cap \Gamma$ is finite and this completes the proof.
\end{proof}

\begin{proof}[Proof of 2. of Theorem \ref{thm.geo.fin}]
Let $\eta_1,\ldots,\eta_k$ be $\Gamma$-bounded parabolic points with disjoint $\Gamma$-orbits and such that the union of their $\Gamma$-orbits exhaust the set of $\Gamma$-bounded parabolic points. By transitivity of $G$ on $\partial T$, let $g_1 \in G$ be such that $g^{-1}_1\xi=\eta_1$ and set $x_1=g_1\Gamma \in X$. By Proposition \ref{prop.compact.orbits} the $G^0_\xi$-orbit of $x_1$ is compact. Since $a$-normalizes $G^0_\xi$, the $G^0_\xi$-orbits of $a^ix_1$ are compact. Let us show that they are disjoint. Since $a$ normalizes $G^0_\xi$, it suffices to show that $a^iG^0_\xi x_1 \subseteq G^0_\xi x_1$ implies $i=0$. Therefore suppose $a^iG^0_\xi x_1 \subseteq G^0_\xi x_1$. This means for every $u \in G^0_\xi$, there exists $u' \in G^0_\xi$ and $\gamma$ such that $a^i u g_1=u' g_1 \gamma$. In other words, $g^{-1}_1 a^{i}u''g_1 \in \Gamma $ for some $u'' \in G^0_\xi$. But $g^{-1}_1 a^{i}u''g_1$ belongs to $G_{g^{-1}_1\xi}$ and since $g^{-1}_1\xi$ is $\Gamma$-bounded parabolic end, by \cite[Proposition 3.1.(i)]{paulin.geo.fin} the intersection $G_{g^{-1}\xi} \cap \Gamma$ is a parabolic group in the sense of Paulin and Bass--Lubotzky \cite[\S 3.1]{paulin.geo.fin} (see also \cite[page 49]{BL}). In particular, it consists of elliptic elements and hence $i=0$.

Moreover, if $g^{-1}\xi$ and $h^{-1}\xi$ are $\Gamma$-bounded parabolic points (say without loss of generality $g^{-1}\xi=\eta_1$ and $h^{-1}\xi=\eta_2$) with disjoint $\Gamma$-orbits, it is easily seen that $G^0_\xi g \Gamma \cap G^0_\xi h \Gamma =\emptyset$. Indeed, otherwise one can find elements $u,u' \in G^0_\xi$ and $\gamma \in \Gamma$ such that $ug=u'h\gamma$. But this means that for some $u'' \in G^0_\xi$, we have $\gamma=h^{-1}u''g$ so that $\gamma \eta_1=h^{-1}u''g \eta_1 =\eta_2$ contradicting the choice of $\eta_j$'s.

Conversely, let $x \in X$ be such that $G^0_\xi x$ is compact. We need to show that $G^0_\xi x=G^0_\xi a^j x_i$ for some $j \in \mathbb{Z}$ and $i=1,\ldots,k$ where $x_i=g_i\Gamma$ with $g_i^{-1}\xi=\eta_i$. By Proposition \ref{prop.compact.orbits}, writing $x=g\Gamma$, we have $g^{-1}\xi$ is a $\Gamma$-bounded parabolic point, so that it is equal to $\gamma \eta_i$ for some $\gamma \in \Gamma$ and $i=1,\ldots,k$. It follows that $g_i \gamma^{-1} g^{-1} \in G_\xi$, and hence it is equal to $a^{-j} u$ for some $j \in \mathbb{Z}$ and $u \in G^0_\xi$. We deduce that $G^0_\xi x=G^0_\xi a^jg_i \Gamma=G^0_\xi a^j x_i$, as required.
\end{proof}

\subsection{Density of non-compact orbits}\label{subsec.density}
Here we prove part of Theorem \ref{thm.geo.fin} concerning density of non-compact $G^0_\xi$-orbits. It will follow from the description of $a$-orbits of $x \in X$ with a non-compact $G^0_\xi$-orbit and a general result (Lemma \ref{lemma.a.rec.implies.density}) relying on mixing properties of the $a$-action.

\begin{lemma}\label{lemma.characterize.a.divergence}
A sequence $a^{-n}x$ diverges to infinity in $X$ if and only if the $G^0_\xi$-orbit of $x$ is compact.
\end{lemma}

\begin{proof}
Suppose $G^0_\xi x$ is compact in $X$. Denote $x=g\Gamma$. By Proposition \ref{prop.compact.orbits}, $g^{-1}\xi$ is a $\Gamma$-bounded parabolic end. To show that $a^{-n}g\Gamma$ diverges to infinity is equivalent to showing that $\tau^{-n}\Gamma$ diverges to infinity, where $\tau=g^{-1}ag$. To see this latter, by continuity, it suffices to show that for some vertex $v \in \Gamma \backslash T$, $p_v(\tau^{-n}\Gamma)$ diverges to infinity in $\Gamma \backslash T$. Let $\tilde{c}$ be a geodesic subray of the translation axis of $\tau$ pointing towards $g^{-1}\xi$. Denote $\tilde{c}(0)=\tilde{v}$. Since $g^{-1}\xi$ is $\Gamma$-bounded 
parabolic, by \cite[Proposition 3.1.(ii)]{paulin.geo.fin}, up to shortening $\tilde{c}$, the restriction of $p:T \to \Gamma \backslash T$ is an isomorphism onto a ray $c$ converging to an 
end $b_{g^{-1}\xi} \in \partial (\Gamma \backslash T)$. Let $c(0)=:v$. By definition of the map $p_v:G/\Gamma \to \Gamma \backslash T$, since $c$ is isomorphic to $\tilde{c}$ and $\tilde{c}$ is a geodesic subray of the translation axis of $\tau$, 
we have $p_v(\tau^{-n}\Gamma)=c(2n)$ (recall that $a$ hence $\tau$ has translation distance $2$). Therefore, $p_v(\tau^{-n}\Gamma) \to b_{g^{-1}\xi}$.

Conversely, suppose $G^0_\xi x$ is not compact. Denote $x=g\Gamma$. By Proposition \ref{prop.compact.orbits} $g^{-1}\xi$ is not a $\Gamma$-bounded parabolic end. Similarly to above, it suffices to show that for some sequence integers $n_k \to + \infty$, $\tau^{-n_k}\Gamma$ belongs to a compact subset of $X$. Since for any vertex $v \in \Gamma \backslash T$, the map $p_v:G/\Gamma \to \Gamma \backslash T$ has compact fibres, it suffices to find such a sequence $n_k$ and a vertex $v \in \Gamma \backslash T$ such that $p_v(\tau^{-n_k}\Gamma)=v_0$ for a vertex $v_0$ of $\Gamma \backslash T$ and for every $k \in \mathbb{N}$. This follows from \cite[Corollaire 3.5]{paulin.geo.fin} as in the proof of Proposition \ref{prop.compact.orbits}.
\end{proof}

The following result provides a criterion for density of non-compact $G_\xi^0$-orbits. We emphasize that in its statement, $\Gamma$ is not assumed to be geometrically finite. Recall that $M=G_{\xi_-}^0 \cap G_\xi^0$.

\begin{lemma}\label{lemma.a.rec.implies.density}
Let $G$ be a non-compact, closed subgroup of $\Aut(T)$ that acts transitively on $\partial T$. Let $\Gamma$ be a lattice in $G$ and $x \in X=G/\Gamma$. Suppose that there exist a compact set $K \subset X$, a sequence of integers $n_k \to + \infty$ such that $a^{-n_k}x \in K$ for every $k \geq 1$. Assume, moreover, that the action of $a$ on $(X,m_X)$ is mixing.  Then, there exists a compact open neighborhood $O^+$ of identity in $G^0_\xi$ such that for every $M$-invariant function $\theta \in C_c(X)$,  up to passing to a subsequence of $n_k$, we have
\begin{equation}\label{eq.dense.1}
\frac{1}{m_{G^0_\xi}(a^{n_k}O^+ a^{-n_k})} \int_{a^{n_k}O^+ a^{-n_k}} \theta(ux) dm_{G^0_\xi}(u) \underset{k \to +\infty}{\longrightarrow} \int \theta(y) dm_X(y).
\end{equation}
\end{lemma}


The proof of the previous lemma is along the same lines as the proof of Proposition \ref{prop.hyp.unique.ergo}. We provide a brief argument.

\begin{proof}[Proof of lemma \ref{lemma.a.rec.implies.density}]
Up to passing to a subsequence of $n_k$ and shrinking $K$, we can find a compact $\hat{K} \subset G$ such that $\hat{K}^{-1}\hat{K} \cap \Gamma =\{\id\}$ and $\pi(\hat{K}) \supseteq K$, where $\pi: G \to X$ is the quotient map. Denote by $\xi_- \in \partial T$ the repelling fixed point of $a$ and let $O^-$ and $O^+$ be compact open symmetric neighborhoods of identity, respectively, in $G^0_{\xi_-}$ and $G^0_\xi$ such that $\hat{K}^{-1}O^-(O^+)^2O^-\hat{K} \cap \Gamma =\{ \id \}$. As in the proof of Theorem \ref{thm.unique.ergo}, for $z \in K$ and $y \in X$ set
\[
\phi_z(y)=\frac{1}{m_X(O^+O^-z)} \mathbbm{1}_{O^+O^-z}(y).
\]
The map $z \mapsto \phi_z$ is continuous as a map from $K$ to $L^2(X,m_X)$. In particular, since $K$ is compact, $\{\phi_z \, | \, z\in K\}$ is a compact subset of $L^2(X,m_X)$.

Let $\theta \in C_c(X)$ and $\varepsilon>0$ be given. Since the $a$-action on $(X,m_X)$ is mixing, for every $z \in K$, we have
\begin{equation}\label{eq.dense.2}
\int_X \phi_z(a^{-n_k} y) \theta(y)dm_X(y) \underset{k \to \infty}{\longrightarrow} \int_X \phi_z(y)dm_X(y)  \int_X \theta(y)dm_X(y),
\end{equation}
where the right-hand side is equal to $\int_X \theta(y) dm_X(y)$ by definition of $\phi_z$. Moreover, since $\{\phi_z \, | \, z\in K\}$ is  compact in $L^2(X,m_X)$, this convergence is uniform in $z \in K$. Let $k \in \mathbb{N}$ be large enough so that the left-hand side is within $\varepsilon$ of $\int_X \theta(y) dm_X(y)$ for every $z \in K$. 

As in $(\ref{eq.unique.ergo5})$, by the choice of $O^-$ and $O^+$ lifting the integral to $G$, and using the product structure the left-hand side of $(\ref{eq.dense.2})$ is equal to
\begin{equation}\label{eq.dense.3}
\frac{1}{m_{G^0_{\xi_-}}(O^-)m_{G^0_\xi}(O^+)} \int\limits_{a^{n_k}O^+ a^{-n_k}} \int\limits_{a^{n_k}O^- a^{-n_k}} \theta(vua^{n_k}z) dm_{G^0_{\xi_-}}(v) dm_{G^0_\xi}(u). 
\end{equation}

By uniform continuity of $\theta$ we can choose a neighborhood of identity $U \subset G_{\xi_-}^0$ such that $|\theta(vw)-\theta(w)|< \varepsilon$ for every $v\in U$ and every $w\in X$. Then for all $k$ large enough 
$a^{n_k}O^-a^{-n_k} \subseteq UM$ (see the proof of Lemma \ref{lem::general_Folner}).  Since $\Delta_{G_\xi}(a)\Delta_{G_{\xi_-}}(a)=1$, where these stand for the corresponding modular functions, by uniform continuity of $\theta$, the quantity in the previous displayed equation is within $\varepsilon$ of $(\ref{eq.dense.4})$, for $k$ large enough and for every $z \in K$:

\begin{equation}\label{eq.dense.4}
\frac{1}{m_{G^0_\xi}(a^{n_k}O^+a^{-n_k})} \int_{a^{n_k}O^+ a^{-n_k}} \theta(ua^{n_k}z) d m_{G^0_\xi}(u).
\end{equation}

Since $a^{-n_k}x \in K$ by choice of $n_k$, $(\ref{eq.dense.1})$ follows by taking $z=a^{-n_k}x$ in the previous equation.
\end{proof}

\begin{proof}[Proof of 1. of Theorem \ref{thm.geo.fin}]
Let $x \in X$ be such that $G^0_\xi x \subset X$ is not compact. By Lemma \ref{lemma.characterize.a.divergence}, there exist a compact set $K \subset X$ and a sequence of integers $n_k \to +\infty$ such that $a^{-n_k}x \in K$ for every $k \geq 1$. 

Let $O$ be a compact open subset of $X$. We apply  
Lemma \ref{lemma.a.rec.implies.density} to the $M$-invariant function $\mathbbm{1}_{MO} \in C_c(X)$, the group $G$ and the hyperbolic element $a$. Note that the action of $a$ is mixing, since $G$ has the Howe--Moore property. We deduce that for some compact open neighborhood $O^+$ of identity in $G^0_\xi$ and up to passing to a subsequence of $n_k$, we have 
$$
\frac{1}{m_{G^0_\xi}(a^{n_k}O^+ a^{-n_k})} \int_{a^{n_k}O^+ a^{-n_k}} \mathbbm{1}_{MO}(ux) dm_{G^0_\xi}(u) \underset{k \to +\infty}{\longrightarrow} m_X(MO).
$$
Since $m_X(MO)>0$, it follows from this convergence that there exists $u \in G^0_\xi$ with $ux \in MO$. Since $M\le G_\xi^0$ this concludes the proof.
\end{proof}


\section{Equidistribution of large compact orbits}\label{sec.equidist}

Here, we prove Proposition \ref{prop.equidist.compact.orbits}. Recall, $G$ is a non-compact, closed, topologically simple subgroup of $\Aut(T)$ acting transitively on $\partial T$ and satisfying Tits' independence property, and $\Gamma$ is a lattice in $G$.

We first record an observation on the $\Gamma$-action on $\partial T$ that gives topological information on the union of $a$-translates of a compact $G_\xi^0$-orbit.
The following result is due to Mostow when $G$ is replaced by a semisimple Lie group with finite center and without compact factors, $G_\xi$ by a parabolic subgroup $P$ and $\partial T$ by the flag manifold $G/P$. Due to similar group decompositions (\S \ref{subsub.decompositions}) and the Howe--Moore property, Mostow's original proof \cite[Lemma 8.5]{mostow} applies mutatis mutandis.


\begin{proposition}\label{prop.mostow}
Let $G$ be a non-compact, closed, topologically simple subgroup of $\Aut(T)$ acting transitively on $\partial T$. Then, for any lattice $\Gamma$ in $G$, the $\Gamma$-action on $\partial T$ is minimal.
\end{proposition}


The previous result implies the $G_\xi$-orbit of $\Gamma$ is dense in $X$. If the  $G^0_\xi$-orbit of $\Gamma$  is compact, the $G_\xi$-orbit of $\Gamma$ is the union of $a^\mathbb{Z}$-translates of the  $G^0_\xi$-orbit of $\Gamma$ (which are themselves compact $G^0_\xi$-orbits), implying this union is dense in $X$. This clearly follows from  Proposition \ref{prop.equidist.compact.orbits} which says furthermore that the $a^\mathbb{Z}$-translates of $G^0_\xi$- orbits themselves get equidistributed with respect to $m_X$ as their volume tends to infinity.



Before proceeding with the proof of Proposition \ref{prop.equidist.compact.orbits}, let us introduce a last notation: when a compact $G^0_\xi$-orbit on $X$ and $x=g\Gamma \in X$ belonging to that orbit are understood, we denote by $m_{\text{orb}}$, the corresponding orbital measure, i.e. the probability measure that is obtained as the push-forward of the Haar probability measure of $G^0_\xi/(G^0_\xi \cap g\Gamma g^{-1})$ by the $G^0_\xi$-equivariant Borel isomorphism that maps this latter onto $G^0_\xi x \subset X$. 

\begin{proof}[Proof of Proposition \ref{prop.equidist.compact.orbits}]
Fix a compact $G^0_\xi$-orbit on $X$ and denote by $m_{\text{orb}}$ the corresponding orbital probability measure on $X$. We need to show that for every $\theta \in C_c(X)$, we have
\begin{equation}\label{eq.equidist.1}
\int \theta(a^iy) dm_{\text{orb}}(y) \underset{i \to +\infty}{\longrightarrow} \int \theta(y) dm_X(y).
\end{equation}

Since $a$ normalizes the compact group $M$ and $M \leq G^0_\xi$ preserves $m_X$, it suffices to show this for an $M$-invariant $\theta \in C_c(X)$ (this reduction is similar to the one in the proof of Theorem \ref{thm.unique.ergo}). Furthermore, it also suffices to show that for every $x \in X$ in the compact orbit and for every small enough neighborhood $O^+$ of identity in $G^0_\xi$, the orbit piece $O^+ x$ equidistributes when translated by $a^{i}$ as $i \to +\infty$, i.e.
\begin{equation}\label{eq.equidist.2}
\frac{1}{m_{G^0_\xi}(O^+)} \int_{O^+x} \theta (a^iy) dm_\text{orb}(y) \underset{i \to +\infty}{\longrightarrow} \int \theta(y) dm_X(y).
\end{equation}

Let $\theta \in C_c(X)$ be an $M$-invariant function, $x \in X$, $O^+$ be a small enough neighborhood of identity in $G^0_\xi$ such that $g \mapsto gx$ is injective on $O^+$. Fix $\varepsilon>0$. Let $O^-$ be a neighborhood of identity in $G^0_{\xi_-}$ small enough such that by $M$-invariance of the uniformly continuous function $\theta$, using $(3)$ of Lemma \ref{lem::some_little_facts}, the left-hand side of \ref{eq.equidist.2} is within $\varepsilon$ of

\begin{equation}\label{eq.equidist.3}
\frac{1}{m_{G^0_{\xi_-}}(a^i O^- a^{-i})} \frac{1}{m_{G^0_\xi}(O^+)} \int_{a^i O^- a^{-i}} \int_{O^+} \theta(va^iux) dm_{G^0_{\xi_-}}(v)dm_{G^0_\xi}(u).
\end{equation}

Changing $v$ to $a^i v a^{-i}$, by product structure of the Haar measure (Corollary \ref{lemma.Haar.VAU}), this is equal to

\begin{equation*}
\frac{1}{m_{G^0_{\xi_-}}(O^-)} \frac{1}{m_{G^0_\xi}(O^+)} \int_{O^-O^+} \theta(a^igx) dm_G(g).
\end{equation*}
Now, up to shrinking $O^-$ and $O^+$ even more so that the map $g \mapsto gx \in X$ is injective on $O^-O^+$, we can write this last equation as

\begin{equation}\label{eq.equidist.5}
\frac{1}{m_{G^0_{\xi_-}}(O^-)} \frac{1}{m_{G^0_\xi}(O^+)} \int_X \theta(a^iy) \mathbbm{1}_{O^- O^+ x}(y) dm_X(y).
\end{equation}

Since by the Howe--Moore property, $a$-action is mixing with respect to $m_X$ and by injectivity and the product structure $m_{G^0_{\xi_-}}(O^-)m_{G^0_\xi}(O^+)=m_X(O^- O^+ x)$, the quantity in $(\ref{eq.equidist.5})$ converges as $i \to +\infty$ to $\int \theta(y) dm_X(y)$, proving $(\ref{eq.equidist.1})$.
\end{proof}








\begin{bibdiv}
\begin{biblist}



\bib{Amann}{thesis}{
author={Amann, Olivier},
 title={Group of tree-automorphisms and their unitary representations},
 note={PhD thesis},
 school={ETH Z\"urich},
 year={2003},
 }



\bib{BEW}{article}{
   author={Banks, Christopher},
   author ={Elder, Murray},
   author ={Willis, George A},
   title={Simple groups of automorphisms of trees
determined by their actions on finite subtrees},
   journal={J. Group Theory},
   volume={18},
   date={2015},
   publisher={de Gruyter},
   pages={235--261},
}

\bib{Bass}{article}{
   author={Bass, Hyman},
   title={Covering theory for graphs of groups},
   journal={J. Pure Appl. Algebra},
   volume={89},
   date={1993},
   number={1-2},
   pages={3--47},
}

\bib{BL}{book}{
   author={Bass, Hyman},
   author={Lubotzky, Alexander},
   title={Tree Lattices},
   series={Progress in Mathematics, vol 176},
   publisher={Birkh\"{a}user Boston},
   date={2001},
}


	
\bib{BW04}{article}{
author={Baumgartner, Udo },
author={Willis, George A.},
title={Contraction groups and scales of automorphisms of totally disconnected locally compact groups},
journal={Israel Journal of Mathematics},
year={2004},
volume={142},
number={1},
publisher={Springer-Verlag},
pages={221-248},
}

\bib{bekka-lubotzky}{article}{
author={Bekka, M. Bachir},
   author={Lubotzky, Alex},
   title={Lattices with and without spectral gap},
   journal={Groups Geometry and Dynamics},
   number={5},
   date={2011},
   pages={251--264}
}

\bib{BeMa}{book}{
   author={Bekka, M. Bachir},
   author={Mayer, Matthias},
   title={Ergodic theory and topological dynamics of group actions on
   homogeneous spaces},
   series={London Mathematical Society Lecture Note Series},
   volume={269},
   publisher={Cambridge University Press, Cambridge},
   date={2000}
}

\bib{benoist.notes}{book}{
  title={R{\'e}seaux des groupes de Lie},
  author={Benoist, Yves},
  series={Notes de cours de M2},
  year={2008}
}


\bib{bowen}{article}{
 title={Unique ergodicity for horocycle foliations},
  author={Bowen, Rufus},
  author={Marcus, Brian},
  journal={Israel Journal of Mathematics},
  volume={26},
  number={1},
  pages={43--67},
  year={1977},
  publisher={Springer}
}



\bib{BPP.book}{book}{
   author={Broise-Alamichel, Anne},
   author={Parkkonen, Jouni},
   author={Paulin, Fr\'{e}d\'{e}ric},
  title={Equidistribution and counting under equilibrium states in
negatively curved spaces and graphs of groups. Applications
to non-Archimedean Diophantine approximation},
   publisher={to appear in Progress in Mathematics, Birkh\"{a}user},
   note ={arXiv:1612.06717}
}


\bib{BPP.cras}{article}{
   author={Broise-Alamichel, Anne},
   author={Parkkonen, Jouni},
   author={Paulin, Fr\'{e}d\'{e}ric},
   title={\'{E}quidistribution non archim\'{e}dienne et actions de groupes sur les
   arbres},
   journal={C. R. Math. Acad. Sci. Paris},
   volume={354},
   date={2016},
   number={10},
   pages={971--975},
}

\bib{broise-paulin.1}{article}{
   author={Broise-Alamichel, Anne},
   author={Paulin, Fr\'{e}d\'{e}ric},
   title={Dynamiques sur le rayon modulaire et fractions continues en
   caract\'{e}ristique $p$},
   journal={J. Lond. Math. Soc. (2)},
   volume={76},
   date={2007},
   number={2},
   pages={399--418},
}

 \bib{BM00b}{article}{
   author={Burger, Marc},
   author={Mozes, Shahar},
   title={Groups acting on trees: from local to global structure},
   journal={Inst. Hautes \'Etudes Sci. Publ. Math.},
   number={92},
   date={2000},
   pages={113--150},
} 

\bib{BM00a}{article}{
   author={Burger, Marc},
   author={Mozes, Shahar},
   title={Lattices in product of trees},
   journal={Inst. Hautes \'{E}tudes Sci. Publ. Math.},
   number={92},
   date={2000},
   pages={151--194},
}



\bib{CaMe13}{article}{
author={Caprace, Pierre-Emmanuel},
author={De Medts, Tom},
title ={Trees, contraction groups, and Moufang sets},
journal = {Duke Mathematical Journal},
number ={13},
pages ={2413--2449},
publisher = {Duke University Press},
volume = {162},
year = {2013},
}



\bib{caprace-reid-willis}{article}{
 title={Locally normal subgroups of totally disconnected groups. Part II: compactly generated simple groups}, 
 volume={5},
 journal={Forum of Mathematics, Sigma}, 
 publisher={Cambridge University Press}, 
 author={Caprace, P-E.},
 author={Reid, C. D. },
 author={Willis, G. A.}, 
 year={2017}, 
 }

\bib{Chou}{article}{
   author={Choucroun, Francis M.},
   title={Analyse harmonique des groupes d'automorphismes d'arbres de
   Bruhat-Tits},
   journal={M\'{e}m. Soc. Math. France (N.S.)},
   number={58},
   date={1994},
   pages={170},
}

\bib{Cio}{article}{
   author={Ciobotaru, C.},
   title={A unified proof of the Howe--Moore property},
   journal={Journal of Lie Theory},
   volume={25},
   date={2015},
   pages={65--89},
   }
   
\bib{CFS.rec}{article}{
  author={Ciobotaru, C.},
  author={Finkelshtein, V.},
  author={Sert, C.},
  title={Quantitative recurrence and equidistribution for horospherical subgroups of groups action on trees},
  journal={In progress},


}

 
   

\bib{dani.horospherical}{article}{
   author={Dani, S. G.},
   title={Invariant measures and minimal sets of horospherical flows},
   journal={Invent. Math.},
   volume={64},
   date={1981},
   number={2},
   pages={357--385},
}
 
 
 
   


\bib{ellis-perrizo}{article}{
   author={Ellis, Robert},
   author={Perrizo, William},
   title={Unique ergodicity of flows on homogeneous spaces},
   journal={Israel J. Math.},
   volume={29},
   date={1978},
   number={2-3},
   pages={276--284},
}


\bib{eskin-mcmullen}{article}{
  author={Eskin, Alex}, 
  author={McMullen, Curt},
  title={Mixing, counting, and equidistribution in Lie groups}, journal={Duke Mathematical Journal},
  
  date={1993},
  volume={71},
  number={1},
  pages={181--209}
}

\bib{federer}{book}{
   author={Federer, Herbert},
   title={Geometric measure theory},
   series={Die Grundlehren der mathematischen Wissenschaften, Band 153},
   publisher={Springer-Verlag New York Inc., New York},
   date={1969},
   pages={xiv+676},
}
	   
   

\bib{furstenberg}{article}{
   author={Furstenberg, Harry},
   title={The unique ergodicity of the horocycle flow},
   conference={
      title={Recent advances in topological dynamics},
   },
   book={
      publisher={Springer, Berlin},
   },
   date={1973},
   pages={95--115},
}

\bib{ghys}{article}{
    author = {Ghys, \'{E}tienne},
     title = {Dynamique des flots unipotents sur les espaces homog\`enes},
      note = {S\'{e}minaire Bourbaki, Vol. 1991/92},
   journal = {Ast\'{e}risque},
    number = {206},
      volume={3},
     pages = {93--136},
}

\bib{Haagerup}{article} { 
  author = {U. Haagerup and A. Przybyszewska},
  title  = {Proper metrics on locally compact groups, and proper affine isometric actions on Banach spaces},
  year   = {2006},
  note   = {arXiv:0606.7964},
}



\bib{halmos}{book}{
  title={Measure theory},
  author={Halmos, Paul R},
  volume={18},
  year={2013},
  publisher={Springer}
}

\bib{hedlund}{article}{
   author={Hedlund, Gustav A.},
   title={Fuchsian groups and transitive horocycles},
   journal={Duke Math. J.},
   volume={2},
   date={1936},
   number={3},
   pages={530--542},
}


\bib{hersonsky-paulin}{article}{
title={Diophantine approximation for negatively curved manifolds},
  author={Hersonsky, Sa'ar},
  author={Paulin, Fr{\'e}d{\'e}ric},
  journal={Mathematische Zeitschrift},
  volume={241},
  number={1},
  pages={181--226},
  year={2002},
  publisher={Springer}
}
 


\bib{khukro}{article}{
  title={Box spaces, group extensions and coarse embeddings into Hilbert space},
  author={Khukhro, A},
  journal={Journal of Functional Analysis},
  volume={263},
  number={1},
  pages={115--128},
  year={2012},
  publisher={Elsevier}
}


\bib{KLNN}{article}{
  title={Farey maps, Diophantine approximation and Bruhat--Tits tree},
  author={Kim, Dong Han},
  author={Lim, Seonhee},
  author={Nakada, Hitoshi},
  author={Natsui, Rie},
  journal={Finite Fields and Their Applications},
  volume={30},
  pages={14--32},
  year={2014},
  publisher={Elsevier}
}




	 
\bib{Knapp}{book}{
   author={Knapp, Anthony W.},
   title={Lie groups beyond an introduction},
   series={Progress in Mathematics},
   volume={140},
   edition={2},
   publisher={Birkh\"{a}user Boston, Inc., Boston, MA},
   date={2002},
   pages={xviii+812},
}	 



	 
 \bib{Lind}{article}{
   author={Lindenstrauss, Elon},
   title={Pointwise theorems for amenable groups},
   journal={Invent. Math.},
   volume={146},
   date={2001},
   number={2},
   pages={259--295},
}

\bib{lubotzky.gafa}{article}{
   author={Lubotzky, A.},
   title={Lattices in rank one Lie groups over local fields},
   journal={Geometric \& Functional Analysis GAFA},
   number={1(4)},
   pages={405--431},
}

\bib{lubotzky-mozes}{book}{
   author={Lubotzky, A.},
   author={Mozes, S.},
   title={Asymptotic properties of unitary representations of tree automorphisms}, 
   series={Harmonic analysis and discrete potential theory},
   date={1992},  
   pages={289--298},
   publisher={Springer, Boston, MA.},

}

\bib{margulis-tomanov}{article}{
author={Margulis, G. A.},
author={Tomanov, G.},
title={Invariant measures for actions of unipotent groups over local fields on homogeneous spaces} ,
journal={Invent. Math.},
volume={116},
number={1-3},
date={1994},
pages={347–-392}
}



\bib{margulis.thesis}{book}{
   author={Margulis, Grigoriy A.},
   title={On some aspects of the theory of Anosov systems},
   series={Springer Monographs in Mathematics},
   publisher={Springer-Verlag, Berlin},
   date={2004},
   pages={vi+139},
}

\bib{mohammadi}{article}{
   author={Mohammadi, Amir},
   title={Measures invariant under horospherical subgroups in positive
   characteristic},
   journal={J. Mod. Dyn.},
   volume={5},
   date={2011},
   number={2},
   pages={237--254},
}

\bib{mostow}{book}{
  title={Strong Rigidity of Locally Symmetric Spaces.(AM-78)},
  author={Mostow, G Daniel},
  volume={78},
  year={2016},
  publisher={Princeton University Press}
}

	
\bib{nagao}{article}{
  title={On $\GL (2,\mathrm{K} [x])$},
  author={Nagao, Hirosi},
  journal={Journal of the Institute of Polytechnics, Osaka City University. Series A: Mathematics},
  volume={10},
  number={2},
  pages={117--121},
  year={1959},
  publisher={Osaka University and Osaka City University, Departments of Mathematics}
}

\bib{paulin.geo.fin}{article}{
    title={Groupes g\'{e}om\'{e}triquement finis d'automorphismes d'arbres et approximation diophantienne dans les arbres},
	author={Paulin, Fr\'{e}d\'{e}ric},
	journal={Manuscripta Mathematica},
	volume={113},
	number={1},
	year={2004},
	pages={1--23}
	}

\bib{paulin}{article}{
 title={Groupe modulaire, fractions continues et approximation diophantienne en caract{\'e}ristique p},
  author={Paulin, Fr{\'e}d{\'e}ric},
  journal={Geometriae Dedicata},
  volume={95},
  number={1},
  pages={65--85},
  year={2002},
  publisher={Springer}
}
 


\bib{paulin-shapira}{unpublished}{
   author={Paulin, Fr\'{e}d\'{e}ric},
   author={Shapira, Uri},
   title={On continued fraction expansions of quadratic irrationals in positive characteristic},
   note={preprint 2018, arXiv:1801.10184},
}

\bib{quint.notes}{unpublished}{
   author={Quint, Jean-Fran\c{c}ois},
   title={Examples of unique ergodicity of algebraic flows},
   note={Lecture Notes, Tsinghua University, Beijing, November 2007},
}


\bib{raghunathan}{book}{
   author={Raghunathan, M. S.},
   title={Discrete subgroups of Lie groups},
   note={Ergebnisse der Mathematik und ihrer Grenzgebiete, Band 68},
   publisher={Springer-Verlag, New York-Heidelberg},
   date={1972},
   pages={ix+227},
}

\bib{ratner.padic}{article}{
   author={Ratner, Marina},
   title={Raghunathan's conjectures for Cartesian products of real and
   $p$-adic Lie groups},
   journal={Duke Math. J.},
   volume={77},
   date={1995},
   number={2},
   pages={275--382},
}

\bib{ratner.unipotent}{article}{
   author={Ratner, Marina},
   title={On measure rigidity of unipotent subgroups of semisimple groups},
   journal={Acta Math.},
   volume={165},
   date={1990},
   number={3-4},
   pages={229--309},
}

\bib{ratner.raghunathan}{article}{
  title={On Raghunathan's measure conjecture},
  author={Ratner, Marina},
  journal={Annals of Mathematics},
  pages={545--607},
  year={1991},
  publisher={JSTOR}
}	

\bib{ratner.joining}{article}{
 title={Horocycle flows, joinings and rigidity of products},
  author={Ratner, Marina},
  journal={Annals of Mathematics},
  pages={277--313},
  year={1983},
  publisher={JSTOR}
}


\bib{sarnak}{article}{
  title={Asymptotic behavior of periodic orbits of the horocycle flow and Eisenstein series},
  author={Sarnak, Peter},
  journal={Communications on Pure and Applied Mathematics},
  volume={34},
  number={6},
  pages={719--739},
  year={1981},
  publisher={Wiley Online Library}
}

\bib{bass-serre}{book}{
  title={Arbres, amalgames, $\SL$2: cours au Coll{\`e}ge de France},
  author={Serre, Jean-Pierre},
  year={1977},
  publisher={Soci{\'e}t{\'e} math{\'e}matique de France}
}

\bib{Ti70}{article}{
   author={Tits, Jacques},
   title={Sur le groupe des automorphismes d'un arbre},
   conference={
      title={Essays on topology and related topics (M\'emoires d\'edi\'es
      \`a Georges de Rham)},
   },
   book={
      publisher={Springer},
      place={New York},
   },
   date={1970},
   pages={188--211},
}


\end{biblist}
\end{bibdiv}
\end{document}